\documentclass[11pt,a4paper,reqno]{amsart}

\newcommand{\pres}[2]{\bigl\langle #1\:|\:#2 \bigr\rangle}

\newcommand{\Gpres}[2]{\operatorname{Gp}\bigl\langle #1\:|\:#2 \bigr\rangle}

\usepackage{tikz}
\usetikzlibrary{decorations.pathreplacing,shapes,arrows}

\usetikzlibrary{decorations.markings}
\usetikzlibrary{arrows,matrix}
\usetikzlibrary{snakes}
\usepgflibrary{arrows}
\tikzset{
    >=stealth',
    punkt/.style={
           rectangle,
           rounded corners,
           draw=black, very thick,
           text width=6.5em,
           minimum height=2em,
           text centered},
    pil/.style={
           ->,
           thick,
           shorten <=2pt,
           shorten >=2pt,}
}
\tikzset{every loop/.style={min distance=2mm,in=225,out=135,looseness=10}}

\tikzset{->-/.style={decoration={
  markings,
  mark=at position #1 with {\arrow{>}}},postaction={decorate}}}

\usepackage{tikz}
\usetikzlibrary{decorations.markings}
\usetikzlibrary{arrows,matrix}
\usepgflibrary{arrows}

\newcommand\GreenL{\mathcal{L}}
\newcommand\GreenR{\mathcal{R}}
\newcommand\GreenH{\mathcal{H}}  
\newcommand\GreenD{\mathcal{D}}
\newcommand\GreenJ{\mathcal{J}}

\usepackage{clrscode}
\usepackage{mathrsfs,amssymb}

\usepackage{amsmath, amsthm}

\newtheorem{theorem}{Theorem}[section]

\newtheorem{question}[theorem]{Question}
\newtheorem{lemma}[theorem]{Lemma}
\newtheorem{corollary}[theorem]{Corollary}

\newtheorem{proposition}[theorem]{Proposition}

\theoremstyle{definition}
\newtheorem{example}[theorem]{Example}
\newtheorem{remark}[theorem]{Remark}
\newtheorem{definition}[theorem]{Definition}

\newcommand{\sgw}{S\Gamma(w)}
\newcommand{\sgo}{S\Gamma(1)}

\newcommand{\gr}{\mathcal{R}}

\newcommand{\roi}{{$\GreenR_1$-injective}}

\begin{document}

\title[Subgroups of Special Inverse Monoids]{Subgroups of $E$-unitary and $\GreenR_1$-injective \\ Special Inverse Monoids}

\maketitle

\begin{center}
    ROBERT D. GRAY\footnote{School of Engineering, Mathematics and Physics, University of East Anglia, Norwich NR4 7TJ, England.
Email \texttt{Robert.D.Gray@uea.ac.uk}.}
\ and \ MARK KAMBITES\footnote{Department of Mathematics, University of Manchester, Manchester M13 9PL, England. Email \texttt{Mark.Kambites@manchester.ac.uk}.

This research was supported by the EPSRC-funded projects EP/N033353/1 `Special inverse monoids: subgroups, structure, geometry, rewriting systems and the word problem' and EP/V032003/1 ‘Algorithmic, topological and geometric aspects of infinite groups, monoids and inverse semigroups’, and by a London Mathematical
Society Research Reboot Grant.
} \\
\end{center}

\begin{abstract}
We continue the study of the structure of general subgroups (in particular maximal subgroups, also known as group $\GreenH$-classes) of special inverse monoids. Recent research of the authors
has established that these can be quite wild, but in this paper we show that if we restrict to special inverse monoids which are \textit{$E$-unitary} (or have a weaker property we call \textit{$\GreenR_1$-injectivity}), the maximal subgroups are strongly governed by the group of units. In particular, every maximal subgroup has a finite index
subgroup which embeds in the group of units. We give a construction to show that every finite group can arise as a maximal subgroup in an \roi\  special inverse monoid with trivial group of units.
It remains open whether every combination of a group $G$ and finite index subgroup $H$ can arise as maximal subgroup and group of units.
\end{abstract}

\section{Introduction}\label{sec_intro}

In this paper we continue the study of so-called \textit{special inverse monoids}: those inverse monoids that admit inverse monoid presentations in which each defining relation has the form $w = 1$. Motivation for this research programme comes both from the beautiful geometric nature of these monoids, and from connections to other areas of semigroup theory and geometric group theory, such as possible applications to the decidability or otherwise of the \textit{one-relator word problem for monoids} (a century-old problem widely regarded as one of the hardest and most important open problems in semigroup theory). To avoid duplication we shall refrain from comprehensive historical discussion, referring instead to the survey article of Meakin \cite{Meakin:2007zt} for the early theory and the introduction to our own recent article \cite{GrayKambitesHClasses} for
subsequent developments.

A key philosophical question about special inverse monoids is the extent to which the structure of the whole monoid is governed by the structure locally around
the identity element (by the units, left and right units or more generally by \textit{Green's $\GreenD$-class} and \textit{$\GreenJ$-class} of $1$). In the case of
special (non-inverse) monoids it is known that the monoid is quite strongly governed in this way: for example all the maximal subgroups are isomorphic
to the group of units \cite{Malheiro2005}. Early results about special inverse monoids suggested that the same kind of relationship might hold, but more recent work has
led to the realisation 
that things are more complex. We recently \cite{GrayKambitesHClasses} studied the possible maximal subgroups (also known as \textit{group $\GreenH$-classes}) which can arise in finitely presented special inverse monoids, answering a question of the first author and Ru\v{s}kuc by showing that the possible groups of units are exactly the finitely generated recursively presented groups, and more generally that the possible group $\GreenH$-classes are exactly the (not necessarily finitely generated) recursively presented groups. This implies in particular that (unlike in the special non-inverse case) the group $\GreenH$-classes are \textit{not} necessarily all isomorphic to the group of units. However, in the examples we constructed it turns out that the group $\GreenH$-classes all \textit{embed} in the group of units, and it is natural to ask if this is always true.

In the present paper we show that this is also not the case: indeed we construct special inverse monoids with trivial group of units and arbitrary finite groups arising as
group $\GreenH$-classes. However, our main theorem is that under a relatively mild (weaker than $E$-unitarity) assumption called \textit{$\GreenR_1$-injectivity},
every $\GreenH$-class is \textit{virtually} embeddable in the group of units, in other words, has a finite index subgroup which embeds in the
group of units.

In addition to this introduction the paper is divided into six sections. Section~\ref{sec_prelim} fixes notation and collects some basic facts about (mostly special) inverse monoids, some of which are folklore but some new and potentially of independent interest. Section~\ref{sec_roi} introduces and studies the new property of $\GreenR_1$-injectivity. Section~\ref{sec_blocks} shows that the Sch\"utzenberger graphs of \roi\ special inverse monoids all admit a certain kind of \textit{block decomposition}, which makes it relatively straightforward to understand their structure modulo the Sch\"utzenberger graph of right units. In Section~\ref{sec_subgroups}
we apply this block decomposition to study maximal subgroups of \roi\ special inverse monoids, in particular establishing that they are constrained to admit a finite index subgroup which embeds in the group of units. Section~\ref{sec_construction} goes some way towards proving the previous theorem ``sharp'', by constructing a family of examples with trivial group of
units but arbitrary finite groups arising as maximal subgroups.
Finally, Section~\ref{sec_nonroi}
digresses slightly to note an interesting fact about special inverse monoids with a generator which is neither a right nor a left unit: in such monoids \textit{every} finite subgroup of the
group of units arises as the maximal subgroup around some idempotent.

\section{Special Inverse Monoids}\label{sec_prelim}

In this section we fix notation and collect together some preliminary results about special inverse monoids; some of these are folklore known to experts but hard to find in the literature, while others (most notably Theorem~\ref{thm_ddesc}) are new and likely to be of independent interest.

Let $A$ be a (typically, but not necessarily, finite) alphabet, and let $A^{\pm 1}$ denote the union of $A$ with a disjoint alphabet $\lbrace a^{-1} \mid a \in A \rbrace$.
We extend the inverse operation to be an involution on $A^{\pm 1}$ by defining $(a^{-1})^{-1} = a$, and to words over $A^{\pm 1}$ by
$(a_1 \dots a_n)^{-1} = a_n^{-1} \dots a_1^{-1}$. For brevity we will often write $w'$ instead of $w^{-1}$. Where $A$ is viewed as a choice of generators for
an inverse monoid $M$, we will sometimes write $\overline{w}$ to denote the element of $M$ represented by a word $w \in (A^{\pm 1})^*$.

Recall that the inverse monoid defined by the presentation $\langle A \mid R \rangle$, 
where $R \subseteq (A^{\pm 1})^* \times (A^{\pm 1})^* $
 is the quotient of the free inverse monoid on $A$ by the congruence generated by 
$R$. All presentations in this paper will be inverse monoid presentations unless stated otherwise. 
An inverse monoid presentation is called \textit{special} if all relations have the form $w=1$, and an inverse monoid is called special if it admits a
special inverse monoid presentation.
We shall assume familiarity with (special) inverse monoids, as well as standard ideas in the
field such as Green's relations, Sch\"utzenberger graphs, Stephen's procedure, $E$-unitarity and the maximal group image. 
The reader unfamiliar with these is directed to
\cite{LawsonBook98} for the classical theory of inverse monoids in general, and \cite{GrayKambitesHClasses} for the more recent theory of special inverse monoids.
Recall that a graph is called \emph{bi-deterministic} if no two edges with the same label share a start vertex or an end vertex.

If $m \in M$ we write $S \Gamma(m)$ for the (right) Sch\"utzenberger graph of $m$, and $\mathcal{H}_m$, $\mathcal{R}_m$ and so forth for the equivalences classes
of $m$ under Green's various relations on $M$. 
If $M$ is an inverse monoid generated by a set $A$ then $S\Gamma(m)$ is a bi-deterministic $(A \cup A^{-1})$-labelled graph. 
We define the root of the graph $S\Gamma(m)$ to be the vertex $mm^{-1}$, that is, the unique vertex in this graph that corresponds to an idempotent element of the inverse monoid $M$. 
We say that a word $w \in (A \cup A^{-1})^*$ can be \emph{be read from the root of $S\Gamma(m)$} if there is a path in $S\Gamma(m)$ labelled by the word $w$ that starts at the vertex $mm^{-1}$.     

\begin{remark}\label{rem:TreeOfBalls} 
The Sch\"{u}tzenberger graphs considered in this paper will usually not be finite. For this reason, when we talk about obtaining these graphs ``using Stephen's procedure'' we are referring to the most general version of Stephen's graphical approach to inverse monoids, as developed in his PhD thesis 
\cite{StephenThesis},
and then later applied in its general form in papers such as  
\cite{Stephen98}
and 
\cite{MMS1990}.
Let us now explain one of the key ways that the results of Stephen will be used in this paper. Full details can be found in 
\cite{StephenThesis}
and also 
\cite{Stephen98}
and 
\cite{MMS1990}.

Let $M = \pres{A}{R}$ be a finitely presented special inverse monoid and let $w \in (A \cup A^{-1})^*$.  As explained in 
\cite[Section 3]{Stephen98}
(from which we adopt the same definitions and notation) 
the Sch\"{u}tzenberger graph $S\Gamma(w)$ can be obtained as the colimit of directed system of graphs obtained by starting with a line $L_w$  labelled $w$ and performing expansions (which involves adding a cycle labelled by a defining relator $r$ at a vertex of a graph where $r$ cannot yet be read) and folding (which involves identifying two directed edges with the same label and the same initial or terminal vertex). The colimit of this system is called the \emph{closed form} of $L_w$ and is denoted $Cl_R(L_w)$. Any object in this system, or colimit of a subsystem, is called an \emph{approximate graph} of $Cl_R(L_w)$. Alternatively an approximate graph can be described in language-theoretic terms; see 
\cite[page 102, Subsection 5.2]{StephenThesis}
for details.  In particular observe that an approximate graph can be infinite. The closed form $Cl_R(L_w)$ of $L_w$ is closed in the sense that no nontrivial expansions or edge foldings can be performed on $Cl_R(L_w)$. It is proved in Stephen's thesis 
\cite[Theorem 5.10]{StephenThesis}
that for any approximate graph $\Gamma$ of $Cl_R(L_w)$ we have $Cl_R(\Gamma) = Cl_R(L_w) = S\Gamma(w)$. In other words, the closed form of any approximate graph of $S\Gamma(w)$ is equal to $S\Gamma(w)$.

Now, let $T$ be the (non-bi-deterministic) infinite graph constructed iteratively by starting with the line $L_w$, and adding a cycle labelled by $r$ at every vertex for every $r \in R$, but not performing any edge folding. It follows from the results of Stephen outlined in the previous paragraph that $S\Gamma(w)$ is obtained by bi-determinising $T$. Indeed, $T$ is an approximate graph for $S\Gamma(w)$ since $T$ is clearly the colimit of the subsystem of all graphs that can be obtained from $L_w$ by only performing expensions and never any foldings. 
(Alternatively one may show that $T$ is an approximate graph $S\Gamma(w)$ by verifying that it satisfies the language-theoretic description of approximate graphs in \cite[page 102, Subsection 5.2]{StephenThesis}.)
Then from the previous paragraph it follows that $Cl_R(T) = S\Gamma(w)$. We claim that the closure $Cl_R(T)$ of $T$ is equal to the bi-determinised form of $T$. Indeed, the collection of graphs $\mathcal{C}$ that can be obtained from $T$ via a finite sequence of expansions or edge foldings is the same as the graphs that can be obtained from $T$ by finite sequence of edge-foldings. This is because no expansions can be applied to $T$ and after a finite sequence of folding is applied to $T$ it is still the case that no expansions can be applied to the resulting graph. Then it follows that the closure of $T$ is the colimit of the collection $\mathcal{C}$ of all partial determinisations of $T$ which is the graph obtained by bi-determinising $T$. But $S\Gamma(w)$ is equal to the colimit of $T$, and so $S\Gamma(w)$ is the graph obtained by bi-determinising $T$. 
Alternatively, one may verify that the the graph obtained by bi-determinising $T$ is clearly closed (in the sense that no non-trivial expansions or foldings can be applied to it) and it is an approximate graph of $S\Gamma(w)$ (since it can be seen to satisfy the language-theoretic description of approximate graphs in \cite[page 102, Subsection 5.2]{StephenThesis})
and it then follows, since the closure of $L_w$ is unique 
\cite[Theorem 3]{MMS87},
that $S\Gamma(w)$ is the graph obtained by bi-determinising $T$. 

As a slight variation on the description given in the previous paragraph, if we begin with $L_w$ and attach a copy of $S\Gamma(1)$ to each vertex of $L_w$, then the resulting graph $\Lambda$ can also be seen to be an approximate graph of $S\Gamma(w)$ and, arguing as in the previous paragraph, the closure of $\Lambda$ is equal to the graph obtained by bi-determinising $\Lambda$, which must equal $S\Gamma(w)$. 
  \end{remark}

The following is well known and easy to prove directly from the definitions.

\begin{proposition}\label{prop_greenschutz}
Let $w$ be a word over the generators for an inverse monoid $M$.
Then $w$ represents:
\begin{itemize}
\item an idempotent element if and only if every path labelled by $w$ in every Sch\"utzenberger graph is a closed path;
\item an element of $\GreenJ_1$ if and only if it labels a path somewhere in $\sgo$;
\item an element of $\GreenR_1$ (a \textit{right unit}) if and only if it labels a path in $\sgo$ starting at the identity;
\item an element of $\GreenL_1$ (a \textit{left unit}) if and only if it labels a path in $\sgo$ ending at the identity; and
\item an element of $\GreenH_1$ if and only if it labels a path in $\sgo$ starting at the identity and a path in $\sgo$ ending at the identity.
\end{itemize}
\end{proposition}
Notably missing from Proposition~\ref{prop_greenschutz} is a description of words representing elements of $\GreenD_1$. For inverse monoids in general there is no easy way to describe these but, surprisingly, in special inverse monoids there is a very nice description akin to those above:
\begin{theorem}\label{thm_ddesc}
Let $M$ be a special inverse monoid generated by $X$, and $w$ a word over $X^{\pm 1}$. Then the following are equivalent:
\begin{itemize}
\item[(i)] $w$ represents an element of $\GreenD_1$;
\item[(ii)] $w$ labels a path in $\sgo$ which passes through (or starts or ends at) the vertex $1$; and
\item[(iii)] there is a decomposition $w = uv$ (as words, possibly empty) where $u$ represents a left unit and $v$ represents a right unit.
\end{itemize}
\end{theorem}

\begin{proof}
The equivalence of (ii) and (iii) is immediate from the characterisation of left units and right units given by Proposition~\ref{prop_greenschutz}. If (iii) holds then since $u$ represents a left unit we have $\overline{u}' \  \overline{u} = 1$ so that
$$\overline{w} \  = \overline{u} \ \overline{v} \  \GreenL \ \overline{u}' \ \overline{u} \ \overline{v} \ = \ \overline{v} \ \GreenR \ 1$$
and (i) holds.

What remains, which is the main burden of the proof and the only part which does not hold for inverse monoids in general, is to show that (i) implies (ii) or (iii).

We do this by showing first that every element in $\GreenD_1$ has \textit{some} representative word with a decomposition of the type in (iii). Then we shall make use of the fact that the inverse monoid defined by an inverse monoid presentation   
$\langle A \mid R \rangle$  
is equal to the monoid defined by the infinite monoid presentation 
\[
  \mathrm{Mon}\langle A, A^{-1} |         R,   \alpha \alpha^{-1} \alpha = \alpha, \alpha \alpha^{-1} \beta \beta^{-1} = \beta \beta^{-1} \alpha \alpha^{-1} (\alpha, \beta \in (A \cup A^{-1})^* \rangle.
 \]
This fact follows from \cite[Theorem 5.10.1]{Howie95}. We shall prove below that if we take any word $w$ that decomposes as in (iii) and apply any one of the infinitely many relations from the infinite monoid presentation above, then word one obtains also admits such a decomposition. Combined with the fact that every every element in $\GreenD_1$ has \textit{some} representative word with a decomposition of the type in (iii), this will suffice to show that every word representing an element in $\GreenD_1$ admits a decomposition of the type in (iii).

For the first step, if $s \in \GreenD_1$ then there exists $t \in M$ with $s \GreenL t \GreenR 1$. Because $s \GreenL t$ we have $s's = t't$, so that
$s = ss's = st't = (st')t$. Now $t$ is a right unit by assumption, and this means we have $(st')'(st') = ts'st'=tt'tt' = 1$, so that $st'$ is a left unit. Thus, if we
choose words representing $st'$ and $t$ respectively, concatenating them will yield a word of the required form representing $(st')t = s$.

Now suppose $w$ is any word which factorises as in (iii), or equivalently that it can be read along some path $\pi$ in $\sgo$ which visits $1$. First note that if we replace a factor of the form $x$ with $xx'x$ or vice versa, or replace a factor $uu'vv'$ with $vv'uu'$ then the resulting word can be read along a path in $\sgo$ starting and ending in the same place as $\pi$ and traversing the same set of edges as $\pi$ (possibly different numbers of times and in a different order). So in particular the resulting word can be read along a path in $\sgo$ which visits $1$. If we insert some relator into $w$ then, since every relator can be read around a closed path at every vertex of $\sgo$, the
resulting path can still be read along a path in $\sgo$ which visits $1$. 

Finally, suppose we remove a factor of $w$ which is a relator, say $w = prq$ where $r$ is a relator and we obtain the word $pq$. Since all paths labelled by relators are closed, we may remove a closed subpath from $\pi$ to obtain a path labelled $pq$. If this path still visits $1$ then we are done. If it does not visit $1$ then the closed subpath we removed from $\pi$ must do so, which means there must be a factorisation $r = ab$ (as words) such that $pa$ labels a path ending at $1$ (so $pa$ represents a left unit) and $bq$ labels a path beginning at $1$ (so $bq$ represents a right unit). Now $a$ represents both a right unit (because it is a prefix of a relator) and a left unit (because it can be read along a path ending at $1$), so $a$ is a unit. But this means $p = (pa)a'$ in $M$, so $p$ represents a left unit, and by
Proposition~\ref{prop_greenschutz} can be read along a path ending at $1$. By a dual argument, $b$ also represents a unit,
so $q = b'(bq)$ in $M$ and $q$ is a right unit and can be read along a path starting at $1$. Thus, the resulting word $pq$ can be read along a path which visits $1$.
\end{proof}

\begin{corollary}\label{cor_ddesc}
A generator in a special inverse monoid presentation represents an element of $\GreenD_1$ if and only if represents an element of $\GreenR_1 \cup \GreenL_1$.
\end{corollary}
\begin{proof}
If $x$ is a generator and $x \GreenD 1$ then by Theorem~\ref{thm_ddesc} it must label some path in $S\Gamma(1)$ passing through $1$; since the label is a single letter such a path must have length
$1$, and therefore must start and/or end at $1$, which means that $x$ is $\GreenR$-related and/or $\GreenL$-related to $1$.
\end{proof}

\begin{remark} 
While it is a general and well-known fact about inverse monoids (which is implicitly proved in the third paragraph of the proof of Theorem~\ref{thm_ddesc}) that every element of $\GreenD_1$ can be decomposed in the monoid as the product of a left and a right unit, it is far more unusual and surprising that every \textit{word} representing such an element can be decomposed \textit{as a word} into words representing a left and a right unit. This behaviour is very particular to special inverse presentations, and can fail even for
non-special presentations of special inverse monoids. For example the non-standard presentation $\textrm{Inv}\langle p, q, r \mid pq = 1, qp = r \rangle$ for the
(special, bisimple) bicyclic monoid contains a generator $r$ representing an element of $\GreenD_1 \setminus (\GreenL_1 \cup \GreenR_1)$; if the theorem held then
$r$ would have to label a path through $1$, which since it is a single letter means a path starting or ending at $1$, but in this case it would represent an element of
$\GreenL_1$ or $\GreenR_1$.
\end{remark}

\begin{proposition}
Any special inverse monoid in which $\GreenR_1 = \GreenH_1$ (or $\GreenL_1 = \GreenH_1$) decomposes as the free product of a group with a free inverse monoid.
\end{proposition}
\begin{proof}
Suppose $M$ is a special inverse monoid in which $\GreenR_1 = \GreenH_1$, the case $\GreenL_1 = \GreenH_1$ being dual.

We claim that in fact $\GreenJ_1 = \GreenH_1$. Indeed, suppose
not for a contradiction. Notice first that there must be a generator in $\GreenJ_1 \setminus \GreenH_1$; indeed if not then every word either contains a
generator outside $\GreenJ_1$ (in which case it does not represent an element of $\GreenJ_1$, since $M \setminus \GreenJ_1$ is an ideal) or has all generators in $\GreenH_1$ (in which
case it represents an element of $\GreenH_1$, since $\GreenH_1$ is a subgroup). Clearly in order to be in $\GreenJ_1$ this generator must appear in a relator, $r$ say. Write $r = uxv$ where $x$ is the
leftmost generator not in $\GreenH_1$. Then the factor $u$ represents a right unit, which since $\GreenR_1 = \GreenH_1$ means it represents a unit. Now in the
monoid we have $xv = u^{-1} u xv = u^{-1} 1 = u^{-1}$, so $xv$ represents a unit. But this means $x$ is right invertible, so $x \in \GreenR_1 = \GreenH_1$, giving a contradiction.

Now it is easy to see that the generating set can be split into generators in $\GreenJ_1 = \GreenH_1$, which generate the group $\GreenH_1$, and
generators not in $\GreenJ_1$ which do not appear in any relation and hence generate a free factor.
\end{proof}

As a consequence we obtain a very simple proof of the following well-known fact:

\begin{corollary}
Every finite special inverse monoid is a finite group.
\end{corollary}
\begin{proof}
It is well known that finite monoids satisfy $\GreenR_1 = \GreenH_1$. 
Indeed, given $x \in \GreenR_1$ write $xy=1$, then since the monoid is finite we have $x^i = x^{i+j}$ for some $i,j \geq 1$ which, right multiplying by $y^i$, gives $x^j=1$ for some $j \geq 1$. Hence $x \in \GreenH_1$.     
So by the above any finite special inverse monoid is free product of
a group with a free inverse monoid. Since non-trivial free inverse monoids are infinite, the free inverse monoid must be trivial and hence the given monoid
is a (necessarily finite) group.
\end{proof}

We shall need the following fact, which is established by the argument in the proof of \cite[Proposition 4.2]{Ivanov:2001kl}.
\begin{proposition}\label{prop_r1gen}
In a special inverse monoid the (non-inverse) submonoid $\GreenR_1$ [respectively $\GreenL_1$] is generated (under multiplication only) by the set of elements represented
by the proper prefixes [suffixes] of the defining relators.
\end{proposition}

We shall need an elementary lemma about bi-deterministic graphs, which is well-known to experts:
\begin{lemma}\label{lemma_uniquemap}
A morphism of connected, bi-deterministic graphs is uniquely determined by where it takes any single vertex. In particular, non-identity automorphisms of such
graphs are fixed-point free. If two connected, bi-deterministic
graphs with distinguished root vertices admit root-preserving morphisms between them in both directions, then the morphisms are isomorphisms.
\end{lemma}
\begin{proof}
Suppose $f : X \to Y$ is a morphism of connected, bi-deterministic graphs. Let $v$ be a vertex of $X$. For each other vertex $u \in X$, because $X$ is connected we
may choose a path from $v$ to $u$, say with label $w$. But now $f(u)$ must be at the end of a path labelled $w$ starting at $f(v)$; because $Y$ is bi-deterministic
there can be only one such path, and so $f(u)$ is determined by $f(v)$. In particular, if $X = Y$ and $f$ is an automorphism fixing a vertex $v$ then it must be the
identity automorphism.

Now if there are root-preserving morphisms $f : X \to Y$ and $g : Y \to X$ then the compositions $f \circ g : X \to X$ and $g \circ f : Y \to Y$ are morphisms
which agree with the identity maps on $X$ and $Y$ on their root vertices; since the identity maps are also morphisms, by the previous paragraph
$f \circ g$ and $g \circ f$ must be equal to the respective identity maps, so $f$ and $g$ are isomorphisms.
\end{proof}

We shall also need the following lemma, which is proved by a similar technique to \cite[Lemma 3.3]{GrayKambitesHClasses}.
\begin{lemma}\label{lemma_morphismfromschutz}
Let $M = \langle A \mid R \rangle$ be a special inverse monoid and suppose $\Omega$ is a rooted bi-deterministic 
$A$-labelled graph such that some word $w \in (A^{\pm 1})^*$ can be read from the root, and every defining relation of $M$ can be read around a closed
path at every vertex of $\Omega$. Then there is a morphism from $S\Gamma(w)$ to $\Omega$, taking the root to the root.
\end{lemma}
\begin{proof}
Let $T$ be the (non-bi-deterministic) infinite graph constructed iteratively by starting with a line labelled $w$, and adding a cycle labelled by $r$ at every vertex for every $r \in R$, but not performing any edge folding. 
We view each of these cycles as oriented in such a way that the word $r$ is the label of the path given by reading the cycle clockwise.    
By a \emph{proper subpath of a cycle of $T$} we mean a path $\pi$ with initial vertex being the vertex at which the cycle was attached in the construction of $T$, and such that $\pi$ is a simple path which traverses the cycle clockwise but does not visit every vertex of the cycle, i.e. the end vertex of $\pi$ is not equal to the start vertex of $\pi$. Note that if $r \in R$ is the label of a cycle in $T$ then any proper subpath of this cycle is labelled by a proper prefix of the word $r$.         
From the construction it follows that for every vertex $u$ of $T$ there is a unique sequence $(\pi_0, \pi_1, \pi_2, \ldots, \pi_k)$ where $\pi_0$ is a simple path starting at the root and traversing
part of the line labelled $w$, each $\pi_i$ for $i \geq 1$ is a proper subpath of a cycle and $\pi_0 \pi_1 \ldots \pi_k$ is a path from the root of $T$ to $u$.      
We define a map from the vertex set of $T$ to vertices in $\Omega$ where the vertex $u$ with corresponding sequence $(\pi_0, \pi_2, \ldots, \pi_k)$ of proper subpaths of cycles maps to the vertex in $\Omega$ obtained by following the path labelled by $p_0 p_1 \ldots p_k$ starting at the vertex $w$ of $\Omega$, where $p_i$ is the label of the path $\pi_i$ for $0 \leq i \leq k$.   
This gives a well-defined (by uniqueness of the sequences of proper subpaths of cycles) map from the vertices of $T$ to the vertices of $\Omega$.   
As a consequence of the assumptions that $\Omega$ is bi-deterministic and every relator from $R$ can be read from every vertex in $\Omega$, this map extends uniquely to a morphism of graphs which maps edges of $T$ to the edges of $\Omega$. 
Let us use $\phi$ to denote this graph morphism from $T$ to $\Omega$.  

\begin{sloppypar}
As explained in 
Remark~\ref{rem:TreeOfBalls}, 
it follows from results of Stephen that $S\Gamma(w)$ is obtained by bi-determinising $T$. 
We claim that $\phi$ induces a well-defined graph morphism from $S\Gamma(w)$ to $\Omega$.       
To see this note that two vertices $v$ and $u$ of $T$ are identified in $S\Gamma(w)$ if and only if there is a path in $T$ between these vertices labelled by a word that freely reduces to the empty word in the free group. Since $\phi$ is a morphism it follows that there is a path in $\Omega$ between $\phi(v)$ and $\phi(u)$ labelled by the same word that freely reduces to the empty word in the free group. Since the graph $\Omega$ is bi-deterministic it follows that $\phi(v) = \phi(u)$. Hence $\phi$ induces a well-defined map from the vertices of $S\Gamma(w)$ to the vertices of $\Omega$.          
Two edges $e$ and $f$ of $T$ are identified in $S\Gamma(w)$ if and only if they have the same label, say $a \in A$, and their start vertices $v$ and $u$ are identified in $S\Gamma(1)$. 
But we have already seen that this means that $\phi(v) = \phi(u)$ which, since $\Omega$ is bi-deterministic means that both $e$ and $f$ must be mapped to the unique edge in $\Omega$ with start vertex $\phi(v) = \phi(u)$ and labelled by $a$.                
This shows that $\phi$ induces a well defined map from the edges of $S\Gamma(w)$ to the edges of $\Omega$.    
\end{sloppypar}

It remains to verify that $\phi$ induces a morphism of graphs from $S\Gamma(w)$ to $\Omega$.    
Let $e$ be an edge in $S\Gamma(w)$. Choose an edge $f$ in $T$ such that $f$ is equal to $e$ when $T$ is bi-determinised, that is, $f$ is a member of the equivalence class of edges that represented $e$. Since $\phi$ is a morphism from $T$ to $\Omega$ it follows that the start vertex of $f$ in $T$ maps to the start vertex of $\phi(f)$ in $\Omega$, and the end vertex of $f$ in $T$ maps to the end vertex of $\phi(f)$ in $\Omega$. 
But by definition $\phi(e) = \phi(f)$ and $\phi$ maps the start vertex $e$ to the same place as the start vertex of $f$, and similarly for the end vertices. 
It follows that $\phi$ induces a morphism of graphs from $S\Gamma(w)$ to $\Omega$. 
\end{proof}

\section{$\mathcal{R}_1$-injectivity}\label{sec_roi}

In this section we define a new property called \textit{$\GreenR_1$-injectivity} which is weaker than $E$-unitarity, establish some of its basic properties and give examples to show that it encompasses many special inverse monoids of interest which are not $E$-unitary.

\begin{definition}
We say that an inverse monoid is \textit{$\GreenR_1$-injective} if the morphism to the maximal group image is injective when restricted to the $\GreenR$-class of the identity. 
\end{definition}

The following result lists a number of equivalent characterisations; as well as being useful later, we hope they help to convince the reader that the definition is natural.
\begin{proposition}\label{roiequiv}
Let $M$ be an inverse monoid generated by a set $X$. Then the following are equivalent:
\begin{itemize}
\item[(i)] $M$ is $\GreenR_1$-injective;
\item[(ii)] $\sgo$ naturally embeds in the Cayley graph of the maximal group image $M / \sigma$;
\item[(iii)] for some $x \in \GreenD_1$, the morphism from $M$ to $M / \sigma$ is injective when restricted to $\GreenR_x$;
\item[(iv)] for every $y \in \GreenD_1$, the morphism from $M$ to $M / \sigma$ is injective when restricted to $\GreenR_y$;
\item[(v)] $\sigma^{-1}(1) \cap \GreenD_1 = E(M) \cap \GreenD_1$;
\item[(vi)] every path in $\sgo$ labelled by a word representing the identity in $M / \sigma$ is closed;
\item[(vii)] the morphism to the maximal group image is injective when restricted to the $\GreenL$-class of the identity;
\item[(viii)] for some $x \in \GreenD_1$, the morphism from $M$ to $M / \sigma$ is injective when restricted to $L_x$;
\item[(ix)] for every $y \in \GreenD_1$, the morphism from $M$ to $M / \sigma$ is injective when restricted to $L_y$.
\end{itemize}
\end{proposition}
\begin{proof}
The equivalence of (i) and (ii) and the fact that (i) implies (iii) are immediate from the definitions.

Suppose (iii) holds for some $x \in \GreenD_1$ and let $y \in \GreenD_1$. Then $x \GreenD y$ so there exists $z$ with $x \GreenR z \GreenL y$. In particular we may write
$z = qy$ for some $q \in M$, and by Green's lemma
there is a bijection
$$\lambda_q : \GreenR_y \to \GreenR_{qy} = \GreenR_z = \GreenR_x, \ \ i \mapsto qi.$$
Now if $a, b \in \GreenD_y$
with $\sigma(a) = \sigma(b)$ then $\sigma(qa) = \sigma(qb)$ where $qa, qb \in \GreenR_x$, so $qa = qb$, which since $\lambda_q$ is a bijection means that $a = b$.
Thus, (iv) holds.

Now suppose (iv) holds. If $e \in E(M)$ then clearly $\sigma(e) = 1$. Now if $s \in \GreenD_1$ with $\sigma(s) = 1$ then $s \GreenR e$ for some idempotent
$e$, but now $\sigma(e) = 1$ and by assumption $\sigma$ is injective on $\GreenR_s$ so we must have $s = e$ and $s$ is idempotent. Thus, (v) holds.

Suppose (v) holds, and let $x, y \in \GreenR_1$ be such that $\sigma(x) = \sigma(y)$. Then we have $x'y \GreenR x' \GreenL 1$ so that 
 $x'y \GreenD 1$. Moreover, $\sigma(x'y) = \sigma(x)' \sigma(y) = \sigma(x)' \sigma(x) = 1$. Thus we may deduce from (v) that $x'y$ is idempotent. Now we have
$$1 = (xx')(yy') = x(x'y)y' = x(x'y)^2y' = (xx')(yx')(yy') = yx'$$
from which it follows that 
$x' = x'yx'$ and $y = yx' y$ which implies that the inverse of $x'$ is $y$ 
and so 
$x=(x')'=y$. Thus, (i) holds.

If (ii) holds then (vi) follows from the fact that paths in a group Cayley graph labelled by words representing the identity are necessarily closed.

Now suppose (vi) holds and let $x, y \in \GreenR_1$ with $\sigma(x) = \sigma(y)$. Choose words $w_x, w_y$ representing $x$ and $y$ respectively. Then by
Proposition~\ref{prop_greenschutz} there are
paths in $\sgo$ starting at $1$ labelled $w_x$ and $w_y$, so there is a path in $\sgo$ from vertex $x$ to vertex $y$ labelled $w_x' w_y$. Clearly in $M / \sigma$ the word $w_x' w_y$ represents $\sigma(x'y) = \sigma(x)' \sigma(y) = 1$, so by (vi) we have that the path from $x$ to $y$ is closed, which must mean $x=y$. Thus, (i) holds.

The equivalence of (i) and (vii) follows from the facts that they are left/right dual, and that condition (v) and the hypotheses are left/right symmetric. The equivalence
of (vii), (viii) and (ix) is then dual to the equivalence of (i), (ii) and (iii)
\end{proof}

\begin{remark}
There are also, of course, further equivalent conditions which are left/right duals to conditions (ii) and (vi). We omit these as stating them would first require the notion of a left Sch\"utzenberger graph, for which we have no further need here.
\end{remark}

\begin{remark}
The equivalence of conditions (i), (iii) and (iv) in Proposition~\ref{roiequiv} can also be deduced from the fact that Sch\"utzenberger graphs of $\GreenR$-classes in the same
$\GreenD$-class are isomorphic \cite[Theorem~3.4(a)]{Stephen:1990ss} together with the fact that group Cayley graphs are homogeneous. 
\end{remark}

\begin{remark}
Notwithstanding the equivalent characterisations given by Proposition~\ref{roiequiv}, we do not expect $\GreenR_1$-injectivity to be a useful or interesting property for inverse monoids in
general, since it gives information only about the $\GreenD$-class
of $1$, and there is no reason to suppose that this has any influence on the wider structure of the monoid. For
example, the condition would be trivially satisfied in any inverse
monoid whose identity is adjoined. But in special inverse monoids,
where the structure of the whole monoid is more strongly influenced by the right and left units, it seems
to be a very powerful property, as we shall see later.
\end{remark}

In view of condition (v) in Proposition~\ref{roiequiv} it is natural to ask if $\GreenR_1$-injectivity is also equivalent to the condition $\sigma^{-1}(1) \cap \GreenR_1 = E(M) \cap \GreenR_1$, in other words $\sigma^{-1}(1) \cap \GreenR_1 = \lbrace 1 \rbrace$. In fact while this condition is self-evidently necessary for $\GreenR_1$-injectivity, it is not sufficient even in special inverse monoids, as the following example shows:
\begin{example}\label{ExGrab4}
Consider the special inverse monoid
$$\langle a, b, c, d \mid acb = adb = cc' = dd' = 1 \rangle.$$
By Proposition~\ref{prop_r1gen} the right units are generated by the proper prefixes of the relators, which since $ac = ad$ in the monoid means by the set $X = \lbrace a, ac, c, d \rbrace$.
The maximal group image is the group with the same presentation, which is equivalent as a group presentation to $\langle a,b,c,d \mid d = c, b = (ac)^{-1} \rangle$, in other words, a free group generated by $a$ and $c$ (with $d$ mapping to $c$ and $b$ to $(ac)^{-1}$). Clearly no positive word over the set $X$ represents the identity in this group, so no non-idempotent right unit of the monoid maps to $1$ in the maximal group image, in other words, we have 
$\sigma^{-1}(1) \cap \GreenR_1 = E(M) \cap \GreenR_1$. On the other hand, the right units $c$ and $d$ get identified, but it can be seen (for example by constructing $S\Gamma(1)$; see Figure~\ref{FigGrab4}) that they are distinct in the monoid, so the monoid cannot be \roi.
Note that the element $c'd$ is a
non-idempotent in $\GreenD_1$ which maps to $1$ in the maximal group image, witnessing the failure of condition (v) in Proposition~\ref{roiequiv}.
\end{example}

%
% Picture of the graph \Omega constructed in the proof below 
%
\begin{figure}
\begin{center}
\begin{tikzpicture}
\tikzstyle{lightnode}=[circle, draw, fill=black!20,
                        inner sep=0pt, minimum width=4pt]
%
%}
\draw[dashed] (6, 0) -- (6.5,0);
\draw[dashed] (-0.5, -2) -- (-0.75,-3);
\draw[dashed] (-0.5, -2) -- (-0.25,-3);
\draw[dashed] (0.5, -2) -- (0.25,-3);
\draw[dashed] (0.5, -2) -- (0.75,-3);
\draw[dashed] (-0.5, -2) -- (-0.5+0.5, -2);
\draw[dashed] (0.5, -2) -- (0.5+0.5, -2); 
\draw[dashed] (-0.5+3.5,   -2+1) --    (-0.75+3.5,  -3+1);
\draw[dashed] (-0.5+3.5,   -2+1) --    (-0.25+3.5,  -3+1);
\draw[dashed] (-0.5+3.5,   -2+1) -- (-0.5+0.5+3.5,  -2+1);
\draw[dashed]  (0.5+2.5+3, -2+1) --     (0.25+2.5+3,-3+1);
\draw[dashed]  (0.5+2.5+3, -2+1) --     (0.75+2.5+3,-3+1);
\draw[dashed]  (0.5+2.5+3, -2+1) --  (0.5+0.5+2.5+3,-2+1); 
\node[lightnode] (A) at (0, 0) {};
\node[lightnode] (B) at (3, 0) {};
\node[lightnode] (B2) at (3, -1) {};
\node[lightnode] (C) at (6, 0) {};
\node[lightnode] (C2) at (6, -1) {};
\node[lightnode] (E) at (-0.5, -2) {};
\node[lightnode] (F) at (0.5, -2) {};
\draw[->-=0.6, thick] (A) -- (E) node[midway,left] {$c$};
\draw[->-=0.6, thick] (A) -- (F) node[midway,right] {$d$};
\draw[->-=0.6, thick] (A) -- (B) node[midway,above left] {$a$};
\draw[->-=0.6, thick] (B) -- (C) node[midway,above left] {$a$};
\draw[->-=0.6, thick] (B2) -- (A) node[midway,below left] {$b$};
\draw[->-=0.6, thick] (C2) -- (B) node[midway,below left] {$b$};
\draw [->-=0.6, thick, out=-45,in=45,looseness=0.75] (C) to node[right]{$d$}  (C2);
\draw [->-=0.6, thick, out=-135,in=135,looseness=0.75] (C) to node[left]{$c$}  (C2);
\draw [->-=0.6, thick, out=-45,in=45,looseness=0.75] (B) to node[right]{$d$}  (B2);
\draw [->-=0.6, thick, out=-135,in=135,looseness=0.75] (B) to node[left]{$c$}  (B2);
\end{tikzpicture}
\end{center}
\caption{
An illustration of $S\Gamma(1)$ for the inverse monoid   
$\langle a, b, c, d \mid acb = adb = cc' = dd' = 1 \rangle$ considered in Example~\ref{ExGrab4}. 
The full graph is obtained by repeatedly gluing this building block freely at every vertex. 
}\label{FigGrab4}
\end{figure}

One might also ask if $\GreenR_1$-injectivity is equivalent to the stronger (than condition (v) in Proposition~\ref{roiequiv}) condition that $\sigma^{-1}(1) \cap \GreenJ_1 = E(M) \cap \GreenJ_1$; we shall see in Example~\ref{roiexample} below that this is not the case.

\begin{proposition}\label{eunitaryroi}
Every $E$-unitary inverse monoid is \roi. 
\end{proposition}
\begin{proof}
It is well-known (indeed, sometimes even taken as a definition) that an inverse monoid (or semigroup) is E-unitary if and only if the pre-image of the identity 
in the maximal group image is exactly the set of
idempotents \cite[Proposition~5.9.1]{Howie95}, and it follows immediately that every E-unitary monoid satisfies condition (v) of Proposition~\ref{roiequiv}.
\end{proof}

The converse of Proposition~\ref{eunitaryroi} is very far from being true, even in very restricted cases.  The following examples show that it can fail even for
positive, special, finite presentations, and for one-relator special presentations.

\begin{example}
The inverse monoid $\langle a,b,c,d \mid acb = adb = 1 \rangle$
is shown in \cite[Section~3]{Ivanov:2001kl} not to be $E$-unitary but is \roi. Indeed, the right units are just the submonoid generated by $a$ and $ac$ (since 
by Proposition~\ref{prop_r1gen} the right units in any special inverse monoid are generated by the proper prefixes of the relators, and it follows from the presentation that $ac = ad$ in the monoid). The maximal group image is the group with the same presentation, which is equivalent as a group presentation to $\langle a,b,c,d \mid c = d, b = ac^{-1} \rangle$, in other words, a free group generated by $a$ and $c$ (with $d$ mapping to $c$ and $b$ to $ac^{-1}$). Since $a$ and $ac$ are easily seen to generate a free monoid inside
this free group, distinct right units in the monoid must map to distinct elements of the maximal group image.
\end{example}

\begin{example}\label{roiexample}
Similarly, the one-relator inverse monoid $\langle x, y \mid xyx' = 1 \rangle$ (a homomorphic image of that in the previous example with $a$, $b$, $c$ and $d$ mapping
respectively to $x$, $x'$, $y$ and $y$) is \roi\ but not E-unitary. In this case the right units are just $1$ and the
positive powers of $x$ (because $xy = x$ in the monoid), and the presentation as a group presentation is equivalent to $\langle x, y \mid y=1 \rangle$, so the maximal group image is an infinite cyclic group generated by $x$ with $y$ mapping to $1$. Again, distinct right units in the monoid map to distinct elements of the maximal group image, and so the monoid is \roi. On the other hand, $y$ is not idempotent in the monoid (as can be seen be constructing $S\Gamma(y)$ and applying the criterion for
idempotency given by Proposition~\ref{prop_greenschutz}) but maps to $1$ in the maximal group image, so
the monoid is not $E$-unitary. Note also that $y \GreenJ 1$, so this example shows that $\GreenR_1$-injectivity is strictly weaker than the condition
$\sigma^{-1}(1) \cap \GreenJ_1 = E(M) \cap \GreenJ_1$.
Finally, notice that in this example the Sch\"utzenberger graph $\sgo$ does not embed into the Cayley graph of the maximal group image as a \textit{full} subgraph: indeed the Cayley graph contains a loop at $1$ labelled $y$, while $\sgo$ does not; see Figure~\ref{FigWorth}.  
\end{example}

\begin{figure}
\begin{center}
\begin{tikzpicture}[scale=0.8]
\tikzstyle{lightnode}=[circle, draw, fill=black!20,
                        inner sep=0pt, minimum width=4pt]
%
%}
\node[lightnode] (A) at (0, 0) {};
\node[lightnode] (B) at (2, 0) {};
\node[lightnode] (C) at (4, 0) {};
\node[lightnode] (D) at (6, 0) {};
\node            (E) at (8, 0) {};
\draw[->-=0.52,shorten <=0pt,shorten >=0pt, thick](B.center)arc(-180-90:180-90:0.4) node[midway,below left] {$y$};
\draw[->-=0.52,shorten <=0pt,shorten >=0pt, thick](C.center)arc(-180-90:180-90:0.4) node[midway,below left] {$y$};
\draw[->-=0.52,shorten <=0pt,shorten >=0pt, thick](D.center)arc(-180-90:180-90:0.4) node[midway,below left] {$y$};
\node[lightnode] (B') at (2, 0) {};
\node[lightnode] (C') at (4, 0) {};
\node[lightnode] (D') at (6, 0) {};
\draw[->-=0.6, thick] (A) -- (B) node[midway,below ] {$x$};
\draw[->-=0.6, thick] (B) -- (C) node[midway,below ] {$x$};
\draw[->-=0.6, thick] (C) -- (D) node[midway,below ] {$x$};
\draw[dashed] (D) -- (E);
\end{tikzpicture}
\end{center}
\caption{\label{FigWorth}
The Sch\"utzenberger graph $S\Gamma(1)$ of the inverse monoid $\langle x, y \mid xyx' = 1 \rangle$ considered in Example~\ref{roiexample}.  
}
\end{figure}

The converse of Proposition~\ref{eunitaryroi} \textbf{does} hold for cyclically reduced (in particular for positive) one-relator special inverse
monoids for the trivial reason that these are known \cite{Ivanov:2001kl} \textbf{all} to be E-unitary!

While the above examples show that $\GreenR_1$-injectivity is quite common, there are also many special inverse monoids which are not $\GreenR_1$-injective. Indeed, the following examples show that even a one-relator special inverse monoid can fail to be \roi, even if the relator is a reduced (but not cyclically reduced, since this is
known to imply $E$-unitarity \cite{Ivanov:2001kl}) word.
\begin{example}\label{example_yYxyX}
The inverse monoid $\langle x, y \mid yy' x y x' = 1\rangle$ is not \roi.
Indeed, the element $y$ is right invertible (since it is a prefix of the defining relator), is easily seen to map to the identity in the maximal group image, but is not equal to $1$ in the monoid. The latter point can be seen in $S\Gamma(1)$ (see Figure~\ref{fig3})
where $y$ labels both a loop and a non-loop edge. The fact that it labels a loop means it must map to $1$ in the maximal group image, and the fact it labels a non-loop edge means condition (iv) of Proposition~\ref{roiequiv} fails. 
\end{example}

\begin{figure}\label{fig3}
\begin{center}
\begin{tikzpicture}[scale=0.8]
\tikzstyle{lightnode}=[circle, draw, fill=black!20,
                        inner sep=0pt, minimum width=4pt]
%
%}
\node[lightnode] (A2) at (0, 2) {};
\node[lightnode]             (A3) at (0, 4) {};
\node            (A4) at (1, 4) {};
\node            (A5) at (0, 5) {};
\draw[->-=0.6, thick] (A) -- (A2) node[midway,left] {$y$};
\draw[->-=0.6, thick] (A2) -- (A3) node[midway,left] {$y$};
\draw[dashed] (A3) -- (A4);
\draw[dashed] (A3) -- (A5);
\node[lightnode] (a) at (0, 2) {};
\node[lightnode] (b) at (2, 2) {};
\node[lightnode] (c) at (4, 2) {};
\node[lightnode] (d) at (6, 2) {};
\node            (e) at (8, 2) {};
\draw[->-=0.52,shorten <=0pt,shorten >=0pt, thick](b.center)arc(-180-90:180-90:0.4) node[midway,below left] {$y$};
\draw[->-=0.52,shorten <=0pt,shorten >=0pt, thick](c.center)arc(-180-90:180-90:0.4) node[midway,below left] {$y$};
\draw[->-=0.52,shorten <=0pt,shorten >=0pt, thick](d.center)arc(-180-90:180-90:0.4) node[midway,below left] {$y$};
\node[lightnode] (b') at (2, 2) {};
\node[lightnode] (c') at (4, 2) {};
\node[lightnode] (d') at (6, 2) {};
\draw[->-=0.6, thick] (a) -- (b) node[midway,below ] {$x$};
\draw[->-=0.6, thick] (b) -- (c) node[midway,below ] {$x$};
\draw[->-=0.6, thick] (c) -- (d) node[midway,below ] {$x$};
\draw[dashed] (d) -- (e);
\node[lightnode] (A) at (0, 0) {};
\node[lightnode] (B) at (2, 0) {};
\node[lightnode] (C) at (4, 0) {};
\node[lightnode] (D) at (6, 0) {};
\node            (E) at (8, 0) {};
\draw[->-=0.52,shorten <=0pt,shorten >=0pt, thick](B.center)arc(-180-90:180-90:0.4) node[midway,below left] {$y$};
\draw[->-=0.52,shorten <=0pt,shorten >=0pt, thick](C.center)arc(-180-90:180-90:0.4) node[midway,below left] {$y$};
\draw[->-=0.52,shorten <=0pt,shorten >=0pt, thick](D.center)arc(-180-90:180-90:0.4) node[midway,below left] {$y$};
\node[lightnode] (B') at (2, 0) {};
\node[lightnode] (C') at (4, 0) {};
\node[lightnode] (D') at (6, 0) {};
\draw[->-=0.6, thick] (A) -- (B) node[midway,below ] {$x$};
\draw[->-=0.6, thick] (B) -- (C) node[midway,below ] {$x$};
\draw[->-=0.6, thick] (C) -- (D) node[midway,below ] {$x$};
\draw[dashed] (D) -- (E);
\end{tikzpicture}
\end{center}
\caption{
The Sch\"utzenberger graph $S\Gamma(1)$ of the inverse monoid $\langle x, y \mid yy' x y x' = 1\rangle$ considered in Example~\ref{example_yYxyX}. }
\end{figure}

\begin{example}\label{example_xyxyXYX}
The inverse monoid $\langle x, y \mid (xyx) y (x' y' x') = 1\rangle$ is not \roi, although the relator is a reduced word (but of course not cyclically reduced, which by
\cite{Ivanov:2001kl} would imply the monoid was $E$-unitary). Indeed, constructing $S\Gamma(1)$ (see Figure~4) we see that $y$ once again labels both loop and non-loop edges, so again $y$ represents the identity in the maximal group image and condition (iv) of Proposition~\ref{roiequiv} fails. 
\end{example}

\begin{figure}\label{fig4}
\begin{center}
\begin{tikzpicture}[scale=0.8]

\tikzstyle{lightnode}=[circle, draw, fill=black!20,
                        inner sep=0pt, minimum width=4pt]
%
%}
\node[lightnode] (A1) at (0, -2) {};
\node[lightnode] (A2) at (0, 2) {};
\node            (A3) at (0, 3) {};
\draw[->-=0.6, thick] (A1) -- (A) node[midway,left] {$x$};
\draw[->-=0.6, thick] (A) -- (A2) node[midway,left] {$x$};
\draw[dashed] (A2) -- (A3);
\node[lightnode] (a) at (0, 2) {};
\node[lightnode] (b) at (2, 2) {};
\node[lightnode] (c) at (4, 2) {};
\node[lightnode] (d) at (6, 2) {};
\node            (e) at (8, 2) {};
\draw[->-=0.52,shorten <=0pt,shorten >=0pt, thick](c.center)arc(-180-90:180-90:0.4) node[midway,below left] {$y$};
\draw[->-=0.52,shorten <=0pt,shorten >=0pt, thick](d.center)arc(-180-90:180-90:0.4) node[midway,below left] {$y$};
\node[lightnode] (b') at (2, 2) {};
\node[lightnode] (c') at (4, 2) {};
\node[lightnode] (d') at (6, 2) {};
\draw[->-=0.6, thick] (a) -- (b) node[midway,below left] {$y$};
\draw[->-=0.6, thick] (b) -- (c) node[midway,below left] {$x$};
\draw[->-=0.6, thick] (c) -- (d) node[midway,below left] {$x$};
\draw[dashed] (d) -- (e);
\node[lightnode] (A) at (0, 0) {};
\node[lightnode] (B) at (2, 0) {};
\node[lightnode] (C) at (4, 0) {};
\node[lightnode] (D) at (6, 0) {};
\node            (E) at (8, 0) {};
\draw[->-=0.52,shorten <=0pt,shorten >=0pt, thick](C.center)arc(-180-90:180-90:0.4) node[midway,below left] {$y$};
\draw[->-=0.52,shorten <=0pt,shorten >=0pt, thick](D.center)arc(-180-90:180-90:0.4) node[midway,below left] {$y$};
\node[lightnode] (B') at (2, 0) {};
\node[lightnode] (C') at (4, 0) {};
\node[lightnode] (D') at (6, 0) {};
\draw[->-=0.6, thick] (A) -- (B) node[midway,below left] {$y$};
\draw[->-=0.6, thick] (B) -- (C) node[midway,below left] {$x$};
\draw[->-=0.6, thick] (C) -- (D) node[midway,below left] {$x$};
\draw[dashed] (D) -- (E);
\end{tikzpicture}
\end{center}
\caption{
The Sch\"utzenberger graph $S\Gamma(1)$ of the inverse monoid $\langle x, y \mid (xyx) y (x' y' x') = 1\rangle$ considered in Example~\ref{example_xyxyXYX}.}
\end{figure}

\section{Sch\"utzenberger Graphs and Blocks}\label{sec_blocks}

Throughout this section $M$ will be an \roi\ special inverse monoid generated by a (not necessarily finite) set $X$. Our aim is to show that under this assumption, the different Sch\"utzenberger graphs of $M$ admit a decomposition into (not necessarily disjoint) copies of
$\sgo$. The existence of such a decomposition stems in large part from the following elementary result, which extends to the \roi\ case an observation made in the $E$-unitary case by Stephen \cite[Theorem~3.8]{Stephen:1990ss}.

\begin{lemma}\label{sgoembeds}
Let $M$ be an \roi\ inverse monoid generated by a set $X$, and
$w$ a word. Then for every vertex $v$ of $S \Gamma(w)$ there is
an injective morphism from  $\sgo$ to $\sgw$ which takes $1$ to $v$.
\end{lemma}
\begin{proof}
Every word readable from $1$ in $\sgo$ is right invertible and therefore readable
from $v$ in $\sgw$. Moreover, if two words $x$ and $y$ reach the same vertex
when read from $1$ in $\sgo$ then they represent the same element, and therefore also reach the same place when read from $v$ in $\sgw$. Thus, there is a well-defined morphism $f : \sgo \to \sgw$ given by setting $f(u)$ to be the unique vertex at the end of a path starting at $v$ and labelled $x$, where $x$ is any
word labelling a path from $1$ to $u$.

For injectivity, suppose $r, s \in \mathcal{R}_1$ are such that $f(r) = f(s)$. This means that $vr = vs$ in the monoid, so in particular $(vr) \sigma (vs)$ which since $\sigma$ is group congruence implies $r \sigma s$. But $M$ is \roi, so this means $r=s$.
\end{proof} 

We shall now build upon Lemma~\ref{sgoembeds} to establish a block decomposition for the lower Sch\"utzenberger graphs.

\begin{definition} A \textit{preblock} of an $X$-labelled directed graph $\Gamma$ is a subgraph\footnote{By a subgraph we mean a subset of vertices and a subset of edges between those
vertices; there is no assumption that it is a ``full'' or ```induced'' subgraph containing all edges between the given vertex set, and indeed in general a preblock will
not be a full subgraph.} of $\Gamma$ which is the image of an injective morphism from $S\Gamma(1)$. A \textit{root} of a preblock is a vertex which is the image of $1 \in S \Gamma(1)$ under such a morphism. A \textit{block} of $\Gamma$ is a preblock which is maximal under
inclusion among preblocks.
\end{definition}

\begin{example} 
Consider the monoid $\langle x, y \mid xyx' = 1 \rangle$ from Example~\ref{roiexample} above, and the Sch\"utzenberger graph $S\Gamma(1)$ as illustrated in Figure~\ref{FigWorth}. It is straightforward to see that there is a preblock rooted at every vertex (extending to the right of the vertex in the illustration), but
only the preblock rooted at $1$, in other words the entire graph, is a block.
\end{example}

\begin{remark}
A preblock (or block) does not typically have a \textit{unique} root; indeed it is easy to see that each preblock will have one root for every automorphism of $S\Gamma(1)$.
By 
\cite[Theorem 3.5]{Stephen:1990ss}
this means the roots of each (pre)block are in (non-canonical) bijection with the units of the monoid.
\end{remark}

\begin{remark}
By inclusion of preblocks, we mean of course that the vertices \textbf{and edges} of one preblock are contained in those of the other. \textit{A priori} it might therefore
seem 
possible for the vertex set of one block to be a proper subset of the vertex set of another block, or even for two distinct blocks to have exactly the same vertex set, but in
fact neither of these things can happen provided the graph $\Gamma$ is bi-deterministic. Indeed suppose a block $C$ contains all the vertices of another block $B$. Choose a root vertex $r$ for $B$. Now because $C$ is isomorphic to $\sgo$, Lemma~\ref{sgoembeds} tells us that $C$ contains a subgraph isomorphic to $\sgo$ rooted at $r$.
Since $\Gamma$ is bi-deterministic this means $C$ must contain all the edges of $B$.
\end{remark}

\begin{lemma}\label{truelemma}
Let $M$ be an \roi\ special inverse monoid generated by a set $X$, and $w$ a word over $X^{\pm1}$. Then
$S \Gamma(w)$ has finitely many blocks. Every preblock of $S \Gamma(w)$ lies in a block. Every vertex of $S \Gamma(w)$ lies in at least one block. All but finitely many edges of $S \Gamma(w)$ lie in at least one block. Any edge which does not lie in a block is a cut edge, and is traversed by the path from the root labelled $w$. Every block has a root which is the vertex corresponding to some prefix of $w$.
\end{lemma}

\begin{proof}
Let $\Lambda$ be the directed, labelled graph obtained by starting with the Munn tree of $w$
and gluing a copy of $S \Gamma(1)$ to every vertex. 
As explained in 
Remark~\ref{rem:TreeOfBalls}, 
it follows from results of Stephen 
that $S \Gamma(w)$ may be obtained from
$\Lambda$ by folding (or \textit{bi-determination}, in Stephen's terminology).

Clearly the copies of $\sgo$ glued to the vertices of the Munn tree are a finite set of preblocks of $\Lambda$, containing
all vertices and all but finitely many edges (the missing ones being those of the Munn tree). It follows from Lemma~\ref{sgoembeds}, that
the bi-determination process applied to $\Lambda$ never identifies two vertices within the same preblock; indeed suppose it identified distinct vertices
$v_1$ and $v_2$ within a preblock with root $v$. For $i = 1,2$ let $w_i$ be a word labelling a path from $v$ to $v_i$ within the preblock.
Then starting at the image of $v$ in $S\Gamma(w)$, there will be paths labelled $w_1$ and $w_2$ leading to the same place, which contradicts
the fact that there is an embedded copy of $\sgo$ rooted at $v$. Hence, the image under bi-determination of a preblock in $\Lambda$ is
always a preblock in $S\Gamma(w)$. Let $C$ be the (finite) set of these preblocks in $S\Gamma(w)$, and $B$ the (finite) subset of preblocks in $C$ which are
maximal in $C$ under containment.

We claim that $B$ is the (finite) set of all blocks of $S\Gamma(w)$. 
Clearly, every vertex in $S\Gamma(w)$ lies in some preblock in $B$. Now any other preblock of $S \Gamma(w)$ is rooted at some vertex in $S \Gamma(w)$, and hence at a vertex of some preblock in $B$. By Lemma~\ref{sgoembeds} again, $S \Gamma(1)$ contains a copy of $S \Gamma(1)$ at every vertex, so it follows that every preblock is contained in a preblock in $B$. Since the preblocks in $B$ are defined to be maximal in $C$, and therefore are not contained in each other, it follows that they are maximal among all preblocks, and therefore comprise all blocks of $S \Gamma(w)$.

Since each block in $B$ is the image of a preblock in $\Lambda$ rooted at a Munn tree vertex, and every vertex of the Munn tree can be reached from the identity by
reading some prefix of $w$, it follows that each block has a root which can be reached from $1$ by reading some prefix of $w$ from the root of $S\Gamma(w)$, and which therefore corresponds to a prefix of $w$.

It remains to show that those edges in $S\Gamma(w)$ which do not lie in a block are cut edges, and are traversed by a reading of $w$ from the root. Since every preblock lies in a block, it suffices to show that those edges which do not lie in a preblock are cut edges. As a precursor to this, we claim that if an edge is a cut edge in some graph then its image after any single bi-determination step (and hence, by induction, after finitely many bi-determination steps) either remains a cut edge or is also the image of a non-cut edge. Indeed, suppose that $e$ is a cut edge, and consider the possible effects of a single bi-determination step, that is, of folding two edges together. If neither edge is $e$ then, since the two edges must have a vertex in common, they both lie in the subgraph at the same end of $e$, and it is clear that folding them cannot create a connection to the subgraph at the other end of $e$, so $e$ remains a cut edge. If one is $e$ and the other is another cut edge, then the resulting folded edge remains the only connection between the two end-points of $e$, and hence is still a cut edge.
So the only way $e$ can cease to be a cut edge is by identification with a non-cut edge, as required.

Now suppose that some edge of $S\Gamma(w)$, $e$ say, does not lie in a preblock. Then every preimage of $e$ in $\Lambda$ must be a Munn tree edge. It follows in particular that all such pre-images are traversed by reading $w$ from the root in $\Lambda$, and hence that $e$ itself is traversed by a reading of $w$ from the root in $\sgw$. Now suppose for a contradiction that $e$ is not a cut edge. Let $\pi$ be a path connecting the endpoints of $e$ without traversing $e$. Since every preimage of $e$ in
$\Lambda$ is
a Munn tree edge, every
preimage of $e$ in $\Lambda$ is a cut edge. Let $f$ be one such preimage of $e$ in $\Lambda$. Now the path $\pi$ must be created during folding after some finite number of bi-determination steps, so there is some finite sequence of
bi-determination steps after which $f$ ceases to be a cut edge, having only been identified with other cut edges. But this contradicts the previous paragraph.
\end{proof}

\begin{remark}\label{remark_disjointness}
Lemma~\ref{truelemma} does \textbf{not} say that the blocks are anywhere close to being disjoint! We shall see some cases in which they are disjoint, or close to disjoint in the sense that the intersections are simple to describe, but in general the intersections can be very complicated. It seems that problems concerning the ``lower'' regions ($\GreenD$-classes other than $\GreenD_1$) of an \roi\ special inverse monoid often reduce to understanding the intersections of blocks in its Sch\"utzenberger graphs.
The lemma also does \textbf{not} say that all cut edges in $S\Gamma(w)$ are traversed by a reading of $w$ from the root: in addition to cut edges not lying in a block, there may be cut edges which \textbf{do} lie in blocks, and these are not necessarily traversed by a reading of $w$ from the root.
\end{remark}

\begin{remark}
Under the stronger assumption of $E$-unitarity, ideas similar to those in Lemma~\ref{truelemma} were first introduced (although not made so explicit)
by Stephen \cite{Stephen:1990ss, Stephen93}. In \cite[Section~2.2.5]{LindbladThesis} and \cite{Hermiller:2010bs} it is implicitly suggested (with an incorrect attribution to Stephen \cite{Stephen:1990ss, Stephen93}, who does not actually make
such a claim) that the existence
of such a decomposition suffices to reduce the word problem in the monoid to the problem of deciding whether a given word represents the identity. This line
of reasoning is flawed, since to understand the lower Sch\"utzenberger graphs well enough to solve the word problem requires understanding not only the
internal structure of the blocks (in other words, of $\sgo$) but also the \textit{intersections} of the blocks, which as discussed above may be very complex.
One of the main results of \cite{Hermiller:2010bs}, stating that the word problem is decidable for special inverse monoids with a single \textit{sparse} (see \cite{Hermiller:2010bs} for the definition) relator,
relies upon this claim and therefore cannot be established by the argument given. Rather the paper establishes only that in such monoids it is decidable whether a given word represents the identity. The decidability of the word problem for these monoids remains open, although we conjecture that it is in fact decidable. It may be that
the methods in \cite{Hermiller:2010bs} can be further developed to directly solve the whole of the word problem. Alternatively, it may be possible to establish
(perhaps using methods from \cite{Hermiller:2010bs}) that these monoids have Sch\"utzenberger graphs quasi-isometric to trees, in which case they have
solvable word problem by subsequent work of the first author, Silva and Szak\'acs \cite{GrayNoraPedro2022}.
\end{remark}

The following example shows that a block decomposition as given by Lemma~\ref{truelemma} does \textbf{not} necessarily exist without the
hypothesis of $\GreenR_1$-injectivity, even for one-relator special inverse monoids.

\begin{figure}\label{fige}
\begin{center}
\begin{tikzpicture}[scale=0.8]
\tikzstyle{lightnode}=[circle, draw, fill=black!20,
                        inner sep=0pt, minimum width=4pt]
%
%}
\node[lightnode, label={left:$v$}] (v) at            (0, 0) {};
\node[lightnode, label={left:$x_1$}] (x1) at          (0, 2) {};
\node[lightnode, label={above right:$x_2$}] (x2) at (2, 0) {};
\node[lightnode, label={above right:$y_1$}] (y1) at (2, 2) {};
\node[lightnode, label={below:$y_2$}] (y2) at       (2, -2) {};
\node[lightnode] (z) at                              (4, 0) {};
\draw [->-=0.55, thick, out=-45,in=45,looseness=0.75] (x2) to node[right]{$a$}  (y2);
\draw [->-=0.55, thick, out=135,in=-135,looseness=0.75] (y2) to node[left]{$b$}  (x2);
\draw[->-=0.6, thick] (v) -- (x2) node[midway,below] {$a$};
\draw[->-=0.6, thick] (x1) -- (y1) node[midway,below] {$a$};
\draw[->-=0.6, thick] (y1) -- (x2) node[midway,above right] {$b$};
\draw[->-=0.6, thick] (x2) -- (z) node[midway,below right] {$c$};
\end{tikzpicture}
\end{center}
\caption{
An approximation during Stephen's procedure of $S\Gamma(ac)$ for the monoid considered in Example~\ref{example_aBAabACabc}.
}
\end{figure}

\begin{example}\label{example_aBAabACabc}
Consider the special inverse monoid
$$\langle a,b,c \mid c'abc = ab'a'aba'=1 \rangle \ = \ \langle a,b,c \mid ab'a'aba' \ c'abc = 1 \rangle.$$
The equivalence of the two presentations is because the second relator in the left-hand presentation is an idempotent in the free inverse monoid, so by standard results from the theory of inverse monoid presentations (see for example 
\cite[Lemma 3.3]{Gray2020})
it can be combined into the other relator.

Consider $\sgo$ and $S \Gamma(ac)$ as constructed by Stephen's procedure (see Figure~5).
Note that the former has a simple
path starting at $1$ labelled $ab'a'$. On the other hand, consider the vertex (call it $v$) at the start of the $a$-edge in the Munn tree in $S\Gamma(ac)$. The
path labelled $ab'a'$ here folds up so that it ends in the same place as its prefix labelled $a$; in other words in this monoid $acc'a' \ ab'a' = acc'a' \ a$ even
though $ab'a' \neq a$. This is significant because there is no path coming into $v$ labelled by an element of $\GreenR_1$, which means that vertex $v$ is not in a
block rooted at some other vertex, and it cannot be in a block rooted at itself because the presence of the above relation means there is not an embedded
copy of $\sgo$ rooted at $v$. Hence, $S \Gamma(ac)$ does not have a block decomposition in the above sense (and so, by the contrapositive of Lemma~\ref{truelemma}, the given monoid cannot be \roi).
\end{example}

\section{Automorphisms and Subgroups}\label{sec_subgroups}

In this section we apply the block decomposition developed in Section~\ref{sec_blocks} to study automorphisms of Sch\"utzenberger graphs, and hence group $\GreenH$-classes of \roi\ special inverse monoids.

The following statement is key to the way in which an \roi\ special inverse monoid is governed by $\GreenR_1$.
\begin{theorem}\label{roimaxsubgroups}
Let $M$ be an \roi\ special inverse monoid. Then every subgroup of $M$ has a finite index subgroup which embeds in the group of units.
\end{theorem}
\begin{proof}
Clearly it suffices to show that every \textit{maximal} subgroup of $M$ has a finite index subgroup which embeds in the group of units. Let $G$ be a maximal subgroup of $M$, and $w$ be a word representing the identity element of $G$. Then by \cite[Theorem 3.5]{Stephen:1990ss}, $G$ is the group of labelled digraph automorphisms of the Sch\"utzenberger graph $\sgw$. From here on we consider $G$ acting on
$\sgw$.

Notice that, because the blocks were defined from the isomorphism type of $\sgw$, any automorphism of $\sgw$ must map blocks to blocks, in
the strong sense that its restriction to (both vertices and edges of) any block is an isomorphism to another block. Thus, the action of $G$ on
$\sgw$ induces an action of $G$ on the (finite, by Lemma~\ref{truelemma}) set of blocks of $\sgw$.

Now choose any block $X$ of $\sgw$, and fix a root $r$ for $X$. Let $K$ be the stabiliser of the point $X$ under the action of $G$ on the set of blocks, in other words,
the subgroup of all automorphisms in $G$ which map $X$ to itself. Then  $K$ is a finite index subgroup of $G$. Since $K$ maps $X$ to itself, the action of $K$ on $\sgw$ restricts to an action by automorphisms on $X$. Moreover, this action is faithful: indeed, since $\sgw$ is a connected, labelled and bi-deterministic graph, its automorphisms are fixed-point free, so any
element of $K$ which acts trivially on $X$ must also act trivially on $\sgw$, meaning it must be the identity element of $K$. So $K$ acts faithfully by automorphisms on $X$, which is isomorphic to $\sgo$.  Hence,  $K$ embeds in the automorphism group of $\sgo$, which is isomorphic to the group of units by another application of \cite[Theorem 3.5]{Stephen:1990ss}.
\end{proof}

One might ask whether the stronger statement, that every subgroup actually embeds in the group of units, is true. The following example shows that it is not:
\begin{figure}
\begin{center}
\begin{tikzpicture}
\tikzstyle{lightnode}=[circle, draw, fill=black!20,
                        inner sep=0pt, minimum width=4pt]
\draw[dashed] (8, 0) -- (8.5,0);
\draw[dashed] (4.5+2, -1) -- (4.5+2.5,-1);
\draw[dashed] (4.5+2+4, -1) -- (4.5+2.5+4,-1);
\draw[dashed] (4.5+2, -2) -- (4.5+2.5,-2);
\draw[dashed] (4.5+2+4, -2) -- (4.5+2.5+4,-2);
\node[lightnode] (A) at (0, 0) {};
\node[lightnode] (B) at (4, 0) {};
\node[lightnode] (B2) at (4, -1) {};
\node[lightnode] (C) at (8, 0) {};
\node[lightnode] (C2) at (8, -1) {};
\node[lightnode] (D) at (4.5+2, -1) {};
\node[lightnode] (D2) at (4.5+2, -2) {};
\node[lightnode] (E) at (7.5+3, -1) {};
\node[lightnode] (E2) at (7.5+3, -2) {};
\draw[->-=0.6, thick] (A) -- (B) node[midway,above left] {$x$};
\draw[->-=0.6, thick] (B2) -- (A) node[midway,below left] {$y$};
\draw [->-=0.6, thick, out=-45,in=45,looseness=0.75] (B) to node[right]{$p$}  (B2);
\draw [->-=0.6, thick, out=135,in=-135,looseness=0.75] (B2) to node[left]{$p$}  (B);
\draw[->-=0.6, thick] (B) -- (C) node[midway,above left] {$x$};
\draw[->-=0.6, thick] (C2) -- (B) node[midway,below left] {$y$};
\draw [->-=0.6, thick, out=-45,in=45,looseness=0.75] (C) to node[right]{$p$}  (C2);
\draw [->-=0.6, thick, out=135,in=-135,looseness=0.75] (C2) to node[left]{$p$}  (C);
\draw[->-=0.6, thick] (B2) -- (D) node[midway,above left] {$x$};
\draw[->-=0.6, thick] (D2) -- (B2) node[midway,below left] {$y$};
\draw [->-=0.6, thick, out=-45,in=45,looseness=0.75] (D) to node[right]{$p$}  (D2);
\draw [->-=0.6, thick, out=135,in=-135,looseness=0.75] (D2) to node[left]{$p$}  (D);
\draw[->-=0.6, thick] (C2) -- (E) node[midway,above left] {$x$};
\draw[->-=0.6, thick] (E2) -- (C2) node[midway,below left] {$y$};
\draw [->-=0.6, thick, out=-45,in=45,looseness=0.75] (E) to node[right]{$p$}  (E2);
\draw [->-=0.6, thick, out=135,in=-135,looseness=0.75] (E2) to node[left]{$p$}  (E);
\end{tikzpicture}
\end{center}
\caption{
An illustration of the Sch\"utzenberger graph $\sgo$ of 
$\langle x,p,y \mid xpy = xp'y = 1 \rangle$ from Example \ref{example_roi1}, 
which clearly has trivial automorphism group.
}\label{Figexample_roi1_1}
\end{figure}

\begin{figure}
\begin{center}
\begin{tikzpicture}
\tikzstyle{lightnode}=[circle, draw, fill=black!20,
                        inner sep=0pt, minimum width=4pt]
\node[lightnode] (A) at (0, 0) {};
\node[lightnode] (B) at (4, 0) {};
\node[lightnode] (B2) at (4, -1) {};
\node[lightnode] (C) at (8, 0) {};
\draw[dashed, shorten <=0pt,shorten >=0pt](B.center)arc(90+180:-270+180:1) node[above] {};
\node (a) at (4, 1) {$S\Gamma(1)$};
\draw[dashed, shorten <=0pt,shorten >=0pt](B2.center)arc(90:-270:1) node[above] {};
\node (a) at (4, -2) {$S\Gamma(1)$};
\draw[->-=0.6, thick] (A) -- (B) node[midway,above left] {$x$};
\draw[->-=0.6, thick] (B2) -- (A) node[midway,below left] {$y$};
\draw [->-=0.6, thick, out=-45,in=45,looseness=0.75] (B) to node[right]{$p$}  (B2);
\draw [->-=0.6, thick, out=135,in=-135,looseness=0.75] (B2) to node[left]{$p$}  (B);
\draw[->-=0.6, thick] (B) -- (C) node[midway,above left] {$y$};
\draw[->-=0.6, thick] (C) -- (B2) node[midway,below left] {$x$};
\end{tikzpicture}
\end{center}
\caption{
An illustration of the Sch\"utzenberger graph $S \Gamma (xy)$ of 
$\langle x,p,y \mid xpy = xp'y = 1 \rangle$ from Example \ref{example_roi1}, 
which has automorphism group isomorphic to the cyclic group $\mathbb{Z}_2$ of order two.
}\label{Figexample_roi1_2}
\end{figure}

\begin{example}\label{example_roi1}
Consider the special inverse monoid
$$\langle x,p,y \mid xpy = xp'y = 1 \rangle.$$
Consider the Sch\"utzenberger graphs $\sgo$ and $S \Gamma (xy)$, which are easy to construct
and are illustrated in Figures~\ref{Figexample_roi1_1} and \ref{Figexample_roi1_2}. 
 It is easy to see that the
former has no automorphisms (so the group of units is trivial), while the latter has an automorphism exchanging the vertices at the start and end of the path $xy$ coming from the Munn tree (so there is a subgroup $\mathbb{Z}_2$, which in fact is easily seen to be a maximal subgroup, in the $\GreenD$-class of $xy$).

The maximal group image of this monoid (in other words, the
group given by interpreting the monoid presentation as a group
presentation) is easily seen to be the free product of an infinite cyclic group generated by $x$ and an order-$2$ cyclic group generated by $yx=p$. It is clear from 
considering $\sgo$ that it embeds into the Cayley graph of this group, so that
the monoid is \roi\ by Proposition~\ref{roiequiv}.

On the other hand, this monoid is \textbf{not} E-unitary, since one can verify
by constructing $S\Gamma(p^2)$ that the element $p^2$ is not
idempotent.
\end{example}

In Section~\ref{sec_construction} below, we will construct further examples where the group of units is trivial but arbitrary finite groups arise as maximal subgroups.
The
example above and those in the following section are \roi\ but not $E$-unitary. We have not been able to construct an $E$-unitary special inverse monoid where the maximal
subgroups do not embed in the group of units, nor to prove that such a monoid cannot exist. Similarly, we do not know if there are stronger restrictions on the
behaviour of subgroups in the 1-relator case. In these cases we do not even know if a trivial group of units implies that all subgroups are trivial.
We shall see shortly (Theorem~\ref{thm_disjoint}) that one case in
which subgroups \textbf{do} have to embed in the group of units is for $\GreenD$-classes of an \roi\ special inverse monoid such that the corresponding block decomposition
as given by Lemma~\ref{truelemma} is disjoint (see Remark~\ref{remark_disjointness} above):

In order to prove this, we first observe that 
 in the case of a $\GreenD$-class whose Sch\"utzenberger graph has edges not lying in any block, the block decomposition suffices to prove a very strong statement about the corresponding maximal subgroups: they are necessarily finite (even if the group of units of the monoid is infinite, and even if the monoid is not finitely presented).
\begin{theorem}\label{thm_cutedgefinite}
Let $M$ be an \roi\ special inverse monoid generated by $X$. Let $w$ be a word such that $S\Gamma(w)$ has an edge not contained in any block. Then
the maximal subgroups in the $\GreenD$-class of $w$ are finite.
\end{theorem}
\begin{proof}
By Lemma~\ref{truelemma} the graph $S\Gamma(w)$ has only finitely many edges not contained in blocks. Since the block decomposition is automorphism invariant,
the automorphisms of $S\Gamma(w)$ must permute these edges. Since the action of the automorphism group is fixed-point free it must act faithfully on
this finite set, and therefore must be finite.
\end{proof}

We are now ready to consider the special case in which the block decomposition
as given by Lemma~\ref{truelemma} is disjoint:
\begin{theorem}\label{thm_disjoint}
Let $M$ be an \roi\ special inverse monoid, $w$ a word and suppose the blocks of $S\Gamma(w)$ have pairwise disjoint vertex sets. Then the maximal subgroups
in the $\GreenD$-class of $w$ are isomorphic to the same subgroup of the group of units, and if $w$ does not represent an element of $\GreenD_1$ this subgroup is finite.
\end{theorem}
\begin{proof}
By \cite[Theorem 3.5]{Stephen:1990ss}, the maximal subgroups are all isomorphic to the automorphism group of $\sgw$, which we will denote by $G$.

Consider the (finite, typically non-bi-deterministic) labelled digraph $\Omega$ with vertices the blocks of $\sgw$, and an edge from $X$ to $Y$ labelled $x$ if and only if $\sgw$ has an edge from a vertex of $X$ to a vertex of $Y$ labelled $x$. Since $\sgw$ is connected and every vertex lies in a block, $\Omega$ is connected.
Since the blocks of $\sgw$ are disjoint, edges between distinct blocks cannot lie within blocks, so by Lemma~\ref{truelemma} they are cut-edges. It follows that the edges of $\Omega$ are all cut-edges: in other words, the underlying undirected graph of $\Omega$ is a tree.

Clearly since automorphisms map blocks to blocks, the action of $G$ by automorphisms on $\sgw$ induces an action by labelled digraph automorphisms on $\Omega$. Since $\Omega$ is a finite digraph whose underlying graph is a tree, there is a vertex of $\Omega$ which is fixed by $G$ 
(because by a result of Halin \cite[Lemma 2]{Halin1973} an automorphism of a
finite undirected tree fixes either an edge or a vertex)
and hence a block $X$ of
$\sgw$ which is fixed setwise by the action of $G$. Thus, $G$ can be restricted to act by directed graph automorphisms on the block $X$, and since
non-trivial automorphisms of $\sgw$ are fixed-point free, this action is faithful. Since $X$ is isomorphic to $\sgo$ this means $G$ acts faithfully on $\sgo$,
so embeds in the automorphism group of $\sgo$ which is exactly the group of units.

It remains to show that either $w$ represents an element of $\GreenD_1$ or $G$ is finite.
Consider the set $K$ of edges in $\sgw$ which have exactly one end in $X$. If $K$ is empty then 
$X$
is not connected to any other block, which since $\sgw$ is connected means $\sgw \cong X \cong \sgo$, so by \cite[Theorem~3.4(a)]{Stephen:1990ss}, $w$ represents
an element of $\GreenD_1$.
On the other, if $K$ is non-empty then $G$ is finite by Theorem~\ref{thm_cutedgefinite}.
\end{proof}

Recall that the block decomposition given by Lemma~\ref{truelemma} leaves open the possibility that finitely many edges do not lie in any block. One might ask
if this can really happen; the following example shows that it can.

%%\
%%
%%\

%
%
\begin{figure}
\begin{center}
\begin{tikzpicture}
\tikzstyle{lightnode}=[circle, draw, fill=black!20,
                        inner sep=0pt, minimum width=4pt]
\node[lightnode] (X) at (-1, -1) {};
\node[lightnode] (Y) at (7, -1) {};
\node[lightnode] (A) at (0, 0) {};
\node[lightnode] (B) at (3, 0) {};
\node[lightnode] (B2) at (3, -1) {};
\node[lightnode] (C) at (6, 0) {};
\draw[dashed, shorten <=0pt,shorten >=0pt](B.center)arc(90+180:-270+180:1) node[above] {};
\node (a) at (3, 1) {$S\Gamma(1)$};
\draw[dashed, shorten <=0pt,shorten >=0pt](B2.center)arc(90:-270:1) node[above] {};
\node (a) at (3, -2) {$S\Gamma(1)$};
\draw[dashed, shorten <=0pt,shorten >=0pt](Y.center)arc(90:-270:1) node[above] {};
\node (a) at (3+4, -2) {$S\Gamma(1)$};
\draw[dashed, shorten <=0pt,shorten >=0pt](X.center)arc(90:-270:1) node[above] {};
\node (a) at (3-4, -2) {$S\Gamma(1)$};
%
%
%%%
%
%
\draw[->-=0.6, thick] (A) -- (X) node[midway,above] {$q$};
\draw[->-=0.6, thick] (C) -- (Y) node[midway,above] {$q$};
\draw[->-=0.6, thick] (B) -- (C) node[midway,above left] {$y$};
\draw[->-=0.6, thick] (A) -- (B) node[midway,above left] {$x$};
\draw[->-=0.6, thick] (B2) -- (A) node[midway,below left] {$y$};
\draw [->-=0.6, thick, out=-45,in=45,looseness=0.75] (B) to node[right]{$p$}  (B2);
\draw [->-=0.6, thick, out=135,in=-135,looseness=0.75] (B2) to node[left]{$p$}  (B);
\draw[->-=0.6, thick] (B) -- (C) node[midway,above left] {$y$};
\draw[->-=0.6, thick] (C) -- (B2) node[midway,below left] {$x$};
\end{tikzpicture}
\end{center}
\caption{
An illustration of the Sch\"utzenberger graph $S\Gamma(qq'xyqq')$ of 
$\langle x,p,q,y \mid xpy = xp'y = 1 \rangle$ from Example \ref{example_fproi1}, 
which has automorphism group isomorphic to the cyclic group $\mathbb{Z}_2$ of order two.
Note that from the definitions $S\Gamma(1)$ for this example is isomorphic to the $S\Gamma(1)$ in Figure~\ref{Figexample_roi1_1}. 
The graph $S\Gamma(qq'xyqq')$ has four blocks, rooted at either end of the two edges in the figure labelled $q$. 
The two copies of $S\Gamma(1)$ rooted at either end of the two edges in the figure labelled $p$ are not blocks, but they are both pre-blocks. 
}\label{Figexample_fproi1}.
\end{figure}

\begin{example}\label{example_fproi1}
Consider the inverse monoid
$$\langle x,p,q,y \mid xpy = xp'y = 1 \rangle.$$
Note that it is the free product of the monoid in Example~\ref{example_roi1} with a free inverse monoid of rank $1$, and is
easily seen to be \roi\ with trivial group of units by a similar argument. Consider the Sch\"utzenberger graph $S\Gamma(qq'xyqq')$.
 This has four blocks, rooted at either end of the two edges labelled $q$. The two edges labelled $q$ are
cut edges not contained in any block. 
See Figure~\ref{Figexample_fproi1} for an illustration of $S\Gamma(qq'xyqq')$.
There is clearly an automorphism swapping the two $q$ edges, and since there are no other $q$ edges and automorphisms are fixed-point free, there can be no other automorphisms. So the automorphism group of
$S\Gamma(qq'xyqq')$, and hence the maximal subgroup in the $\GreenD$-class of $qq'xyqq'$, is isomorphic to $\mathbb{Z}_2$.
\end{example}

\section{Subgroups Differing from the Group of Units}\label{sec_construction}

Our aim in this section is to construct examples of \roi\ special inverse monoids where the group of units is trivial but an arbitrary finite group arises as a maximal subgroup. 

\begin{theorem}\label{thm:trivGroupUnits}
For every finite group $G$, there exists an \roi\ special inverse monoid with trivial group of units and a maximal subgroup isomorphic to $G$.
\end{theorem}

We prove the theorem with a construction and a number of lemmas. Let $G$ be a finite group.
For simplicity we consider a presentation for $G$ with very large sets of generators and relations.

Specifically, consider the finite special monoid presentation $\langle A \mid R \rangle$ for $G$ where $A = G \setminus 
\lbrace 1 \rbrace$ and $R$ consists of all 2-letter and 3-letter
positive words over $A$ which are equal to $1$ in $G$. For each generator $a \in A$ introduce new letters
$x_a$ and $y_a$ and their formal inverses, and define $\overline{a} = x_a y_a$. For $w \in (A \cup A^{-1})^*$
define $\overline{w} = \overline{w_1} \dots \overline{w_{|w|}}$.
Let $X = \lbrace x_a \mid a \in A \rbrace$ and $Y = \lbrace y_a \mid a \in A \rbrace$.

For each relator $r = r_1 \dots r_{|r|} \in R$, and each
$1 \leq k \leq |r|$ introduce a new letter $\delta_{r,k}$, let $\Delta$ be the alphabet of these letters, and define a word
$$s_{r,k} = x_{r_k} \delta_{r,k} (\delta_{r,k-1})^{-1} y_{r_{k-1}} \in X \Delta \Delta^{-1} Y$$
where indices are interpreted modulo $|r|$, so that $r_{0} = r_{|r|}$ and $\delta_{r,-1} = \delta_{r,|r|}$.
 Now let $M$ be
the special inverse monoid generated by the set
$$X \cup Y \cup \Delta$$
subject to the set of four-letter relations
$$\lbrace s_{r,k} \mid r \in R, 1 \leq k \leq |r| \rbrace.$$

We shall study the Sch\"utzenberger graphs of the monoid $M$, and other graphs with edges labelled by the generators of $M$. In these graphs, we shall
use the terms \textit{$x$-edges}, \textit{$y$-edges} and \textit{$\delta$-edges} to mean edges labelled respectively by some $x_a$, by some $y_a$ and by some
$\delta_{r,k}$.

\begin{lemma}\label{lemma_sgoprops}
The graph $S\Gamma(1)$ has the following properties:
\begin{itemize}
\item[(i)] The root is the unique vertex with the property that all edges incident with it are either $x$-edges going out or $y$-edges coming in.
\item[(ii)] There are no vertices having both an $x$-edge coming in and a $y$-edge going out.
\end{itemize}
In particular $\sgo$ has trivial automorphism group, so $M$ has trivial group of units.
\end{lemma}
\begin{proof}
We consider the construction of $S\Gamma(1)$ by Stephen's procedure.

Consider first a union of cycles labelled by the defining relations of $S$, amalgamated at the root, without bi-determinising.
The vertices can be divided into four types:
\begin{itemize}
\item[(a)] the root, which has $x$-edges going out and $y$-edges coming in;
\item[(b)] vertices with an $x$-edge coming in and a $\delta$-edge going out;
\item[(c)] vertices with two $\delta$-edges coming in; note that these two edges will have labels of the form $\delta_{r,k}$ and $\delta_{r,k-1}$ (where as usual $k-1$ is interpreted modulo $|r|$) and since $R$ contains no relations of length $1$ these labels are distinct;
\item[(d)] vertices with a $\delta$-edge coming in and a $y$-edge coming out;
\end{itemize}
Determinising $x$-edges in this graph will identify various vertices of type (b), and bi-determinising $y$-edges will identify various vertices of type (d). Let $\Gamma'$ be
the graph resulting from this bi-determinisation. We claim that $\Gamma'$ is now bi-deterministic. Indeed, the only remaining thing which could fail is bi-determinism of the $\delta$-edges. However, it follows from the definition of the relations that each possible $\delta$-edge label appears only twice, once with its start at a vertex of type (b), and once with its start at a vertex of type (d).
Hence, two $\delta$-edges with the same label cannot have the same start vertex. Moreover, each of the $\delta$-edges ends at a vertex of type (c), and we have seen that two $\delta$-edges meeting at a type (c) vertex have differing labels, so two $\delta$-edges with the same label cannot have the same end vertex.

Now notice that vertices of types (b), (c) and (d) in $\Gamma'$ have no $x$-edges coming out, and no $y$-edges coming in, while the root of $\Gamma'$ has only $x$-edges coming out and $y$-edges coming in. It follows that if we attach a new copy of $\Gamma'$ at each non-root vertex of $\Gamma'$, the resulting graphs remains bi-deterministic.

Thus, by Stephen's procedure, $S\Gamma(1)$ can be constructed iteratively as a tree of copies of $\Gamma'$.
It is clear that the original root remains the only vertex which has only $x$-edges going out and $y$-edges coming in (since every other vertex is constructed as type (b), (c) or (d) and therefore has other edges incident with it). Moreover, the only way a vertex can have an $x$-edge coming in is if it is constructed as a type (b) vertex, while the only way it can have
a $y$-edge going out is if it is constructed as a type (d) vertex. Thus, no vertex has both an $x$-edge coming in and a $y$-edge coming out.

Finally, it follows from (i) that automorphisms of $\sgo$ must fix the root, and since automorphisms of bi-deterministic labelled graphs are fixed-point free, this means that the automorphism group is trivial.
\end{proof}

We now define a graph $\Omega$ which has:
\begin{itemize}
\item for each element $g \in G$, a vertex $v_g$;
\item for each element $g \in G$ and generator $a \in A$, a vertex $u_{g,a}$, an edge from $v_g$ to $u_{g,a}$ labelled $x_a$ and an edge from $u_{g,a}$ to $v_{ga}$ labelled $y_a$;
\item for each element $g \in G$ and relation $r = r_1 \dots r_{|r|} \in R$, a new vertex $t_{g,r}$, and for each $1 \leq k \leq |r|$ an edge from
$u_{gr_1 \dots r_{k-1}, r_k}$ to $t_{g,r}$ labelled $\delta_{r,k}$; and 
\item at each vertex $u_{g,a}$ and $t_{g,r}$ an attached copy of $S\Gamma(1)$.
\end{itemize}
We shall refer to the vertex $v_1$ as the \textit{root} of $\Omega$.
We shall show, eventually, that the graph $\Omega$ is isomorphic to a Sch\"utzenberger graph of $M$.

\begin{lemma}
The graph $\Omega$ is bi-deterministic.
\end{lemma}
\begin{proof}
By construction,
\begin{itemize}
\item the edges explicitly created are readily verified to have the required property;
\item \begin{sloppypar} the attached copies of $S\Gamma(1)$ are by definition internally bi-deterministic; and \end{sloppypar}
\item the vertices at which we attached copies of $S\Gamma(1)$ do not have explicitly created $x$-edges going out or $y$-edges coming in, so by Lemma~\ref{lemma_sgoprops} the attachment of $S\Gamma(1)$ at these vertices does not cause any non-determinism. 
\end{itemize}
\end{proof}

\begin{lemma}\label{lem:AutOmega}
The automorphism group of $\Omega$ is isomorphic to $G$.
\end{lemma}
\begin{proof}
It is immediate from symmetry of the definition that there is a faithful action of $G$ where $h$ acts by taking $v_g$ to $v_{hg}$, taking $u_{g,a}$ to $u_{hg,a}$, taking $t_{g,r}$ to $t_{hg,r}$, and extending in the obvious way to permute the attached copies of $S\Gamma(1)$. What remains is to show that there are no more automorphisms
of $\Omega$. Suppose, then, that $f : \Omega \to \Omega$ is an automorphism.

Fix some $a \in A$. Notice that the vertices of the form $u_{g,a}$ are the only vertices with an $x_a$-edge
coming in and a $y_a$-edge going out. Indeed, by construction none of the vertices of the form $v_g$ or $t_{g,r}$ have this property, and by
Lemma~\ref{lemma_sgoprops} the vertices in the attached copies of $\sgo$ do not have this property either. Hence, the set of vertices $\lbrace u_{g,a} \mid g \in G \rbrace$ must be
preserved by $f$. Let $h$ be such that $f(u_{1,a}) = u_{h,a}$. Now $f$ agrees with the action of $h$ on the vertex $u_{1,a}$, so by Lemma~\ref{lemma_uniquemap}
it must act the same as $h$ on the whole graph.
\end{proof}

\begin{lemma}\label{lemma_definingrelationsinomega}
Every defining relation $s_{r,k}$ can be read around a closed path at every vertex in $\Omega$.
\end{lemma}
\begin{proof}
Every vertex except those of the form $v_g$ by definition lies at the root of an attached copy of $S\Gamma(1)$, and hence certainly has the claimed property. For those of the form $v_g$, we can let $s = r_1 \dots r_{k-1}$ where $r = r_1 \dots r_{|r|}$ and
now we have a closed path
$$v_g = v_{(gs^{-1})r_1 \dots r_{k-1}} \xrightarrow{x_{r_k}} u_{g, r_k} \xrightarrow{\delta_{r,k}} t_{gs^{-1},r} \xrightarrow{(\delta_{r,k-1})^{-1}} u_{gs^{-1} r_1 \dots r_{k-2}, r_{k-1}}\xrightarrow{y_k} v_g.$$
(Figure~\ref{fig:SGammaW2} below illustrates a possible example.)
\end{proof}

Now let $W$ be the set of all words over $A$ of length $4$ or less, and define
$$w = \prod_{x \in W}^n (\overline{x}) (\overline{x})^{-1} \in M$$
noting that the order of the product is unimportant because the factors are idempotent, and therefore commute.
Our aim is to show that the Sch\"utzenberger graph $S\Gamma(w)$ is isomorphic to $\Omega$. 

Before proving 
$S\Gamma(w)$ is isomorphic to $\Omega$
we give an example to illustrate the result. 

\begin{example}\label{Ex:R1ConstructionExample} 
Let $G$ be the finite cycle group $\mathbb{Z}_3$ or order three.   
Following the construction outlined above we first write  
a special monoid presentation $\langle A \mid R \rangle$ for $G$ where $A = G \setminus 
\lbrace 1 \rbrace$ and $R$ consists of all 2-letter and 3-letter
words equal to $1$ in $G$. 
So we have 
\[
\langle A \mid R \rangle =
\langle a,b \mid aaa=1, bbb=1, ab=1, ba=1 \rangle.
\]
Hence $R = \{r,t,u,v \}$ where $r=aaa, t=bbb, u=ab, v=ba$. 
Then for this example the definitions above give 
\[
X = \{ x_a, x_b \}, \quad 
Y = \{ y_a, y_b \}, 
\]
and 
\[
\Delta = \{ 
\delta_{r,1}, \delta_{r,2}, \delta_{r,3}, 
\delta_{t,1}, \delta_{t,2}, \delta_{t,3}, 
\delta_{u,1}, \delta_{u,2}, 
\delta_{v,1}, \delta_{v,2}   
\}.
\]
Then $M$ is the special inverse monoid with generating set $X \cup Y \cup \Delta$ and  
the following 
four-letter relations
\[\lbrace s_{q,k} \mid q \in R, 1 \leq k \leq |q| \rbrace\]
which for this example gives the following ten words   
\begin{align*}
s_{r,1}&= x_{r_1} \delta_{r,1} (\delta_{r,3})^{-1} y_{r_{3}}   & 
s_{r,2}&= x_{r_2} \delta_{r,2} (\delta_{r,1})^{-1} y_{r_{1}}   \\
s_{r,3}&= x_{r_3} \delta_{r,3} (\delta_{r,2})^{-1} y_{r_{2}}   & 
s_{t,1}&= x_{t_1} \delta_{t,1} (\delta_{t,3})^{-1} y_{t_{3}}   \\
s_{t,2}&= x_{t_2} \delta_{t,2} (\delta_{t,1})^{-1} y_{t_{1}}   & 
s_{t,3}&= x_{t_3} \delta_{t,3} (\delta_{t,2})^{-1} y_{t_{2}}   \\
s_{u,1}&= x_{u_1} \delta_{u,1} (\delta_{u,2})^{-1} y_{u_{2}}   & 
s_{u,2}&= x_{u_2} \delta_{u,2} (\delta_{u,1})^{-1} y_{u_{1}}   \\
s_{v,1}&= x_{v_1} \delta_{v,1} (\delta_{v,2})^{-1} y_{v_{2}}  & 
s_{v,2}&= x_{v_2} \delta_{v,2} (\delta_{v,1})^{-1} y_{v_{1}}.    
\end{align*}
Here $r_2$ denotes the second letter in the word $r \equiv aaa$ which is $a$, so $x_{r_2} = x_a$, and similarly for the other symbols in the equations above.  
Now, as above, let $W$ be the set of all words over $A$ of length $4$ or less, and define
\[w = \prod_{x \in W}^n (\overline{x}) (\overline{x})^{-1} \in M.\]
Below we shall prove that that the Sch\"utzenberger graph $S\Gamma(w)$ is isomorphic to the graph $\Omega$ defined above. Then by applying Lemma~\ref{lem:AutOmega} it will follow that the automorphism group of $S\Gamma(w)$ is isomorphic to the group we started with, in this case the cyclic group of order three.

A key part of the proof will be that for every relator word $abc$ in the presentation of $M$ we can read the word $\overline{abc} = x_a y_a x_b y_b x_c y_c$ in $S\Gamma(w)$ from every vertex.   
To illustrate the idea for this example, consider the relation $aaa=1$ in the presentation of $G$, that is, the relation word $r$. 
In the presentation of the special inverse monoid $M$ we have the following three relator words 
\begin{align*}
s_{r,1}&= x_{a} \delta_{r,1} (\delta_{r,3})^{-1} y_{a},   & 
s_{r,2}&= x_{a} \delta_{r,2} (\delta_{r,1})^{-1} y_{a},   &
s_{r,3}&= x_{a} \delta_{r,3} (\delta_{r,2})^{-1} y_{a}.   &    
\end{align*}
By definition $aaa \in W$ and so $w$ is a product of idempotent words
with one of the factors of that product being  
$
(x_a y_a x_a y_a x_a y_a)(x_a y_a x_a y_a x_a y_a)^{-1}
$. 
In particular this means that in $S\Gamma(w)$ the word 
$x_a y_a x_a y_a x_a y_a$ 
can be read from the root vertex $ww^{-1}$. 
This is a key difference with $S\Gamma(1)$ where the word 
 $x_a y_a x_a y_a x_a y_a$ 
cannot be read from the root. 
Now if we begin with the path labelled by this word and carry out several steps of Stephen's procedure  
we obtain the graphs given in 
Figure~\ref{fig:SGammaW1}
and
Figure~\ref{fig:SGammaW2}.
We begin by attaching the relator words $s_{r,1}$,
$s_{r,2}$ 
and
$s_{r,3}$ to the path 
$x_a y_a x_a y_a x_a y_a$ as in the proof of Lemma~\ref{lemma_pathtodeltar} below to obtain the first graph in  
Figure~\ref{fig:SGammaW1}. 
Then partially bi-determinising the first graph leads to the second graph in 
Figure~\ref{fig:SGammaW1}. 
Then identifying the two edges labelled $\delta_{r,3}$ gives the graph in 
Figure~\ref{fig:SGammaW2}. 
In this graph $(x_a y_a)^3$ can be read around a closed circuit based at the origin, so the same is true in $S\Gamma(w)$.    
The general proof below that $S\Gamma(w)$ is isomorphic to $\Omega$ uses this idea applied to all the relations in the presentation based at every vertex in the Sch\"utzenberger graph.
\end{example}
\begin{figure} 
%
% Before folding 
%
\begin{center}
\begin{tikzpicture}[scale=0.8]
\tikzstyle{lightnode}=[circle, draw, fill=black!20,
                        inner sep=0pt, minimum width=4pt]

\tikzstyle{darknode}=[circle, draw, fill=black!90,
                        inner sep=0pt, minimum width=4pt]
\node[darknode] (A) at (0, 0) {};
\node[lightnode] (B) at (2, 0) {};
\node[lightnode] (C) at (4, 0) {};
\node[lightnode] (D) at (6, 0) {};
\node[lightnode] (E) at (8, 0) {};
\node[lightnode] (F) at (10, 0) {};
\node[lightnode] (G) at (12, 0) {};
\node[lightnode] (X) at (0, 4) {};
\node[lightnode] (X1) at (-1, 2) {};
\node[lightnode] (X2) at (1, 2) {};
\node[lightnode] (Y) at (0+4, 4) {};
\node[lightnode] (Y1) at (-1+4, 2) {};
\node[lightnode] (Y2) at (1+4, 2) {};
\node[lightnode] (Z) at (0+8, 4) {};
\node[lightnode] (Z1) at (-1+8, 2) {};
\node[lightnode] (Z2) at (1+8, 2) {};
\draw[->-=0.6, thick] (A) -- (X2) node[midway,right] {$x_a$};
\draw[->-=0.6, thick] (X2) -- (X) node[midway,right] {$\delta_{r,1}$};
\draw[->-=0.6, thick] (X1) -- (X) node[midway,left] {$\delta_{r,3}$};
\draw[->-=0.6, thick] (X1) -- (A) node[midway,left] {$y_a$};
\draw[->-=0.6, thick] (C) -- (Y2) node[midway,right] {$x_a$};
\draw[->-=0.6, thick] (Y2) -- (Y) node[midway,right] {$\delta_{r,2}$};
\draw[->-=0.6, thick] (Y1) -- (Y) node[midway,left] {$\delta_{r,1}$};
\draw[->-=0.6, thick] (Y1) -- (C) node[midway,left] {$y_a$};
\draw[->-=0.6, thick] (E) -- (Z2) node[midway,right] {$x_a$};
\draw[->-=0.6, thick] (Z2) -- (Z) node[midway,right] {$\delta_{r,3}$};
\draw[->-=0.6, thick] (Z1) -- (Z) node[midway,left] {$\delta_{r,2}$};
\draw[->-=0.6, thick] (Z1) -- (E) node[midway,left] {$y_a$};
\draw[->-=0.6, thick] (A) -- (B) node[midway,below] {$x_a$};
\draw[->-=0.6, thick] (B) -- (C) node[midway,below] {$y_a$};
\draw[->-=0.6, thick] (C) -- (D) node[midway,below] {$x_a$};
\draw[->-=0.6, thick] (D) -- (E) node[midway,below] {$y_a$};
\draw[->-=0.6, thick] (E) -- (F) node[midway,below] {$x_a$};
\draw[->-=0.6, thick] (F) -- (G) node[midway,below] {$y_a$};
\end{tikzpicture}
\end{center}

\

%
% After some folding  
%
\begin{center}
\begin{tikzpicture}[scale=0.8]
\tikzstyle{lightnode}=[circle, draw, fill=black!20,
                        inner sep=0pt, minimum width=4pt]

\tikzstyle{darknode}=[circle, draw, fill=black!90,
                        inner sep=0pt, minimum width=4pt]
\node[darknode] (A) at (0, 0) {};
\node[lightnode] (B) at (2, 0) {};
\node[lightnode] (C) at (4, 0) {};
\node[lightnode] (D) at (6, 0) {};
\node[lightnode] (E) at (8, 0) {};
\node[lightnode] (F) at (10, 0) {};
\node[lightnode] (G) at (12, 0) {};
\node[lightnode] (X1) at (-2, 0) {};
\node[lightnode] (Y) at (0+4, 4) {};
\draw[->-=0.6, thick] (X1) -- (Y) node[midway,above] {$\delta_{r,3}$};
\draw[->-=0.6, thick] (B) -- (Y) node[midway,right] {$\delta_{r,1}$};
\draw[->-=0.6, thick] (D) -- (Y) node[midway,right] {$\delta_{r,2}$};
\draw[->-=0.6, thick] (F) -- (Y) node[midway,above] {$\delta_{r,3}$};
\draw[->-=0.6, thick] (X1) -- (A) node[midway,below] {$y_a$};
\draw[->-=0.6, thick] (A) -- (B) node[midway,below] {$x_a$};
\draw[->-=0.6, thick] (B) -- (C) node[midway,below] {$y_a$};
\draw[->-=0.6, thick] (C) -- (D) node[midway,below] {$x_a$};
\draw[->-=0.6, thick] (D) -- (E) node[midway,below] {$y_a$};
\draw[->-=0.6, thick] (E) -- (F) node[midway,below] {$x_a$};
\draw[->-=0.6, thick] (F) -- (G) node[midway,below] {$y_a$};
\end{tikzpicture}
\end{center}
\caption{
Two steps in the construction of $S\Gamma(ww^{-1})$ in  Example~\ref{Ex:R1ConstructionExample}.
}\label{fig:SGammaW1}
\end{figure}

\begin{figure} 
%
% After all the folding  
%
\begin{center}
\begin{tikzpicture}[scale=0.8]
\tikzstyle{lightnode}=[circle, draw, fill=black!20,
                        inner sep=0pt, minimum width=4pt]

\tikzstyle{darknode}=[circle, draw, fill=black!90,
                        inner sep=0pt, minimum width=4pt]
\node[darknode] (A) at (0, 0) {};
\node[lightnode] (B) at (3, 0) {};
\node[lightnode] (C) at (6, 0) {};
\node[lightnode] (D) at (4.5, 3) {};
\node[lightnode] (E) at (3, 6) {};
\node[lightnode] (F) at (1.5, 3) {};
\node[lightnode] (M) at (3, 1.7) {};
%%%%
\draw[->-=0.6, thick] (B) -- (M) node[midway,right] {$\delta_{r,1}$};
\draw[->-=0.6, thick] (D) -- (M) node[midway,below] { \; $\delta_{r,2}$};
\draw[->-=0.6, thick] (F) -- (M) node[midway,below] {$\delta_{r,3}$};
%%%%
\draw[->-=0.6, thick] (A) -- (B) node[midway,below] {$x_a$};
\draw[->-=0.6, thick] (B) -- (C) node[midway,below] {$y_a$};
\draw[->-=0.6, thick] (C) -- (D) node[midway,right] {$x_a$};
\draw[->-=0.6, thick] (D) -- (E) node[midway,right] {$y_a$};
\draw[->-=0.6, thick] (E) -- (F) node[midway,left] {$x_a$};
\draw[->-=0.6, thick] (F) -- (A) node[midway,left] {$y_a$};
\end{tikzpicture}
\caption{Illustrates both the third step in the construction of $S\Gamma(ww^{-1})$ in Example~\ref{Ex:R1ConstructionExample}, and also an example case of Lemma~\ref{lemma_definingrelationsinomega}.}
\label{fig:SGammaW2}
\end{center}
\end{figure}

Returning to the general proof, we shall prove that  
$S\Gamma(w)$ is isomorphic to $\Omega$
by showing that there are morphisms (preserving the
root) between these graphs in both directions.

\begin{lemma}\label{lemma_morphismtoomega}
There is a morphism of labelled directed graphs from $S\Gamma(w)$ to $\Omega$, mapping the root to the root.
\end{lemma}
\begin{proof}
We know that
\begin{itemize}
\item[(i)] $\Omega$ is bi-deterministic;
\item[(ii)] $w$ can be read from the root in $\Omega$ (since every word of the form $\overline{x}$ for $x \in A^*$ can be); and
\item[(iii)] by Lemma~\ref{lemma_definingrelationsinomega}, every defining relation can be read around a closed path at every vertex in $\Omega$.
\end{itemize}
The result now follows from Lemma~\ref{lemma_morphismfromschutz}.
\end{proof}

We now proceed to show that there is a morphism in the other direction.

For each relation $r = abc$ let $\Omega_r$ be the subgraph of $\Omega$ consisting of $v_1$, $v_a$, $v_{ab}$, $u_{1,a}$, $u_{a,b}$, $u_{ab,c}$, $t_{1,r}$ and all the edges between them. Similarly for $r = ab$ let $\Omega_r$ be the subgraph of $\Omega$ consisting of $v_1$, $v_a$, $u_{1,a}$, $u_{a,b}$ and $t_{1,r}$ and edges
between them.

\begin{lemma}\label{lemma_pathtodeltar}
With $M$ and $w$ as above, if $r = abc$ [respectively, $r = ab$] is a relation in $R$ and the word $\overline{abc} = x_a y_a x_b y_b x_c y_c$ [respectively, $\overline{ab} = x_a y_a x_b y_b$] is readable at some vertex $v$ of $S\Gamma(w)$, then there is a morphism from $\Omega_r$ to $S\Gamma(w)$ taking the root of $\Omega_r$ to $v$.
\end{lemma}
\begin{proof}
We prove the case for $r$ of length $3$, the length $2$ case being very similar.
Let $v_{1}' = v$, and let $u_{1,a}'$, $v_a'$, $u_{a,b}'$, $v_{ab}'$ and $u_{ab,c}'$ be the vertices reached in $S\Gamma(w)$ on reading $x_a$, $x_a y_a$, $x_a y_a x_b$, $x_a y_a x_b y_b$
and $x_a y_a x_b y_b x_c$ respectively from $v$. Certainly we can read the relators $s_{r,1}$, $s_{r,2}$ and $s_{r,3}$ around closed cycles at $v$, $v_a'$ and $v_{ab}'$ respectively. The fact that the $S\Gamma(w)$ is bi-deterministic means that the vertices reached after reading the first two letters of each of these cycles must be the
same;
call this vertex $t_{1,r'}$. Now it is easy to verify that the map taking each vertex $x$ of $\Omega_r$ to the vertex we have
designated as $x'$ in $S\Gamma(w)$ must be a morphism.
\end{proof}

We now introduce some more notation. Let $z_1$ be the root vertex of $S\Gamma(w)$, and for each $g \in A = G \setminus \lbrace 1 \rbrace$ let $z_g$ denote the vertex in $S\Gamma(w)$ reached by reading $x_g y_g$ from the root.
\begin{lemma}\label{lemma_pathsgowheretheyshould}
Any path in $S\Gamma(w)$ starting at the root and with label of the form $\overline{a_1 \dots a_n}$ where $a_1 \dots a_n = g$ in $G$ ends at $z_g$.
\end{lemma}
\begin{proof}
The claim holds by definition when $n=0$ and $n=1$, so assume for induction that it is true for $k$ and consider a path from the root
with label $\overline{a_1 \dots a_{k+1}}$ which leads to a vertex $v$. By the inductive hypothesis the prefix path labelled $\overline{a_1 \dots a_k}$
leads to $z_h$ where $h = a_1 \dots a_k$ in $G$. 
If $h = 1$ then $a_{k+1} = g$ 
 and there is a path from the root vertex $z_1$ to $v$ labelled $\overline{a_{k+1}}$,
so by the case $n=1$ already established we have $v = z_{a_{k+1}} = z_g$ as required. So assume from now on that $h \neq 1$.

Now if $g = 1$ then $h a_{k+1} = 1$ in $G$, so the two-letter word $r = h a_{k+1}$ is one of the defining relators of $G$. 
By Lemma~\ref{lemma_pathtodeltar} there is a morphism from $\Omega_r$ to $S\Gamma(w)$, taking the root to the root, so the path
from the root labelled $x_h y_h x_{a_{k+1}} y_{a_{k+1}}$ must be closed. But this path ends at $v$, so we must have $v = z_1 = z_g$.

If $g \neq 1$ then it follows from the definition of $w$ that $\overline{h a_{k+1} (g^{-1}})$ is readable from
the root in $S\Gamma(w)$ where $g^{-1}$ is the positive letter of $A$ corresponding to the inverse in $G$ of $A$, so that $h a_{k+1} (g^{-1})$ is a positive
word over $A$. Now by definition we have $ha_{k+1} g^{-1} = 1$ in $G$, so the three-letter word $r = h a_{k+1} (g^{-1})$ is one of the defining relators of $G$. 
By Lemma~\ref{lemma_pathtodeltar} there is a morphism from $\Omega_r$ to $S\Gamma(w)$, taking the root to the root, so this path must be closed.
Moreover the word $gg^{-1}$ is also a defining relator of $G$ and $\overline{gg^{-1}}$ is also readable from the root, and hence by Lemma~\ref{lemma_pathtodeltar} must be read around a closed path. Since the graph is bi-deterministic, it follows that $v = z_g$.
\end{proof}

\begin{lemma}\label{lemma_omegaisomorphism}
The Sch\"utzenberger graph $S\Gamma(w)$ is isomorphic to $\Omega$.
\end{lemma}
\begin{proof}
By Lemma~\ref{lemma_morphismtoomega}, there is a morphism from $S\Gamma(w)$ to $\Omega$, taking the root to the vertex $v_1$, so by Lemma~\ref{lemma_uniquemap} it
will suffice to show that there is also a morphism from
$\Omega$ to $S\Gamma(w)$, taking $v_1$ to the root. We can define such a morphism as follows:
\begin{itemize}
\item Each vertex of the form $v_g$ is mapped to the vertex $z_g$. (In particular, $v_1$ is mapped to the root vertex $z_1$.)
\item Each vertex of the form $u_{g,a}$ is mapped to the vertex at the end of the unique edge labelled $x_a$ leaving $z_g$; by the definition
of $w$ there is an edge in $S\Gamma(w)$ from $g$ to this vertex labelled $x_a$ and by Lemma~\ref{lemma_pathsgowheretheyshould} and the fact that 
$x_g y_g x_a y_a$ is readable from the root, there is also an edge from this vertex to $w_{ga}$ labelled $y_a$.
\item Each vertex of the form $t_{g,r}$ is mapped to the image of the vertex $t_{1,r}$ under the morphism from $\Omega_r$ to
$S\Gamma(w)$ which takes the root to $z_g$; the existence of a morphism from $\Omega_r$ ensures that this vertex has edges
leading to the correct images of vertices of the form $u_{h,a}$.
\item Since $S\Gamma(w)$ is a Sch\"utzenberger graph, for every vertex $v$ in it there is a morphism from $S\Gamma(1)$ to $S\Gamma(w)$ taking the root
to $v$. Each attached copy of $S\Gamma(1)$ in $\Omega$ is attached at some vertex $v$ on which we have already defined our map. We map the whole copy
of $S\Gamma(1)$ to $S\Gamma(w)$ by the morphism which takes the root to the appropriate vertex.
\end{itemize}
At each stage we have verified that the new vertices on which we define the map are sent to vertices which have the required edges to the images of those vertices on which it is already defined; thus, the given map on vertices can be extended to edges to give a morphism as required.
\end{proof}

\begin{lemma}\label{lemma_misroi}
The monoid $M$ is \roi.
\end{lemma}
\begin{proof}
The maximal group image is given by the group presentation 
\[
K = \Gpres{X \cup Y \cup \Delta}{s_{r,k} \ (r \in R, 1 \leq k \leq |r|)}
\]
where 
\[
s_{r,k} = x_{r_k} \delta_{r,k} (\delta_{r,k-1})^{-1} y_{r_{k-1}} \in X \Delta \Delta^{-1} Y
\]
with indices interpreted modulo $|r|$.
Our aim is to identify the group $K$ by performing certain Tietze transformations to simplify the presentation.  

Fix $r \in R$ where $|r|=n$ and consider the set of relators $s_{r,k}$. We can eliminate $\delta_{r,1}$ and the relator $s_{r,1}$ by rearranging the latter as
\[
\delta_{r,1} = x_{r_1}^{-1}  y_{r_{n}}^{-1} \delta_{r,n}, 
\] 
and substituting the right-hand-side in place of $\delta_{r,1}$ in $s_{r,2}$, which is the only other relation in which $\delta_{r,1}$ appears. Then we may eliminate $\delta_{r,2}$ and $s_{r,2}$ using 
\[
\delta_{r,2} 
= x_{r_2}^{-1}  y_{r_{1}}^{-1} \delta_{r,1} 
= x_{r_2}^{-1}  y_{r_{1}}^{-1} x_{r_1}^{-1}  y_{r_{n}}^{-1} \delta_{r,n}.
\] 
and appropriately modifying $s_{r,3}$.
Continuing in this way once we have eliminated all of the generators  
$\delta_{r,1}, \ldots, \delta_{r,n-2}$ our original set of relations will be reduced just to the following two relations 
\begin{eqnarray*} 
\delta_{r,n-1} &=& 
x_{r_{n-1}}^{-1}  y_{r_{n-2}}^{-1}  
x_{r_{n-2}}^{-1}  y_{r_{n-3}}^{-1} \ldots
x_{r_2}^{-1}  y_{r_{1}}^{-1} x_{r_1}^{-1}  y_{r_{n}}^{-1} \delta_{r,n}
\\
1 &=& x_{r_n}^{-1}  y_{r_{n-1}}^{-1} \delta_{r,n-1} \delta_{r,n}^{-1} .  
  \end{eqnarray*}
Finally we eliminate $\delta_{r,n-1}$ using the first of these relations, and substituting into the second we obtain the single relation 
\begin{eqnarray*} 
1 & =  & x_{r_n}^{-1}  y_{r_{n-1}}^{-1} 
x_{r_{n-1}}^{-1}  y_{r_{n-2}}^{-1}  
x_{r_{n-2}}^{-1}  y_{r_{n-3}}^{-1} \ldots
x_{r_2}^{-1}  y_{r_{1}}^{-1} x_{r_1}^{-1}  y_{r_{n}}^{-1} \delta_{r,n}
\delta_{r,n}^{-1} \\
& =  & x_{r_n}^{-1}  y_{r_{n-1}}^{-1} 
x_{r_{n-1}}^{-1}  y_{r_{n-2}}^{-1}  
x_{r_{n-2}}^{-1}  y_{r_{n-3}}^{-1} \ldots
x_{r_2}^{-1}  y_{r_{1}}^{-1} x_{r_1}^{-1}  y_{r_{n}}^{-1}.
  \end{eqnarray*}
By cyclically permuting we see that this relation can be replaced by the relation $\overline{r}=1$ where  
\[
\overline{r} = \overline{r_1 \ldots r_{|r|}} = 
x_{r_1}y_{r_1} \ldots x_{r_n}y_{r_n}. 
\]
Note that the generators $\delta_{r,1}, \ldots, \delta_{r,n-1}$ were all eliminated using these Tietze transformations; the generator $\delta_{r,n}$ was not eliminated but no longer features in any relations, and will therefore generate a free factor. 

For each $r \in R$ define $\lambda_r = \delta_{r,|r|}$ and set $\Lambda = \{\lambda_r \mid r \in R \}$. Repeating the above sets of Tietze transformations for every $r \in R$ we obtain the following presentation for the maximal group image
\[
K = 
\Gpres{X \cup Y \cup \Lambda}{ \overline{r}=1 \ (r \in R)}
=
FG(\Lambda) \ast \Gpres{X \cup Y}{\overline{r}=1 \ (r \in R)}
\]
Now for each $a \in A$ add a new redundant generator $z_a$ and relation $z_a = x_ay_a$ to the presentation and set $Z = \{z_a \mid a \in A \}$.    
Since $y_a = {x_a}^{-1}z_a$ we can then eliminate the redundant generators $\{y_a \mid a \in A\}$ from the presentation giving   
\begin{align*}
K & \cong 
FG(\Lambda) \ast 
\Gpres{X \cup Z}{\tilde{r}=1 \ (r \in R)} \\
& \cong 
FG(\Lambda \cup X) \ast \Gpres{A}{R} \cong  
FG(\Lambda \cup X) \ast G 
\end{align*}
where for $r = a_{i_1} \ldots a_{i_k}$ we define  
$\tilde{r} = z_{a_{i_1}} \ldots z_{a_{i_k}}$. 

We now move on to proving that the monoid $M$ is \roi. By Proposition~\ref{prop_r1gen} the submonoid $\GreenR_1$ is generated as a monoid
by the proper prefixes of the defining relators $s_{r,k}$.
Hence every element of $\GreenR_1$ can be written as a product of the elements 
\[
x_{r_k}, \quad 
y_{r_k}^{-1}, \quad  
x_{r_k}\delta_{r,k}
\]  
where $r \in R$, $1 \leq k \leq |r|$.   
    We compute the image of each of these generators in the maximal group image:
\[
K = \Gpres{\Lambda \cup X \cup Z}{\tilde{r}=1 \textrm{ for all } r \in R}  
\]
where $\Lambda = \{ \lambda_r = \delta_{r,|r|} \mid r \in R \}$
and $Z = \{ z_a = x_ay_a \mid a \in A \}$:
\begin{itemize}
\item The image of $x_{r_k}$ is $x_{r_k}$.
\item The image of $y_{r_k}^{-1}$ is $z_{r_k}^{-1} x_{r_k}$. 
\item The image of 
$x_{r_{|r|}}\delta_{r,|r|}$ is 
$x_{r_{|r|}}\lambda_r$. 
\item For $r \in R$ and $1 \leq k < |r| = n$ the image of  
$x_{r_k}\delta_{r,k}$ 
is  
\begin{align*} 
x_{r_k}\delta_{r,k} 
&= 
x_{r_k}x_{r_k}^{-1}
y_{r_{k-1}}^{-1}x_{r_{k-1}}^{-1}
\ldots
y_{r_1}^{-1}x_{r_1}^{-1}
y_{r_n}^{-1}\delta_{r,n} \\
&= 
y_{r_{k-1}}^{-1}x_{r_{k-1}}^{-1}
\ldots
y_{r_1}^{-1}x_{r_1}^{-1}
y_{r_n}^{-1}\delta_{r,n} \\
&= 
z_{r_{k-1}}^{-1}
\ldots
z_{r_1}^{-1}
z_{r_n}^{-1} x_{r_n}
\delta_{r,n}  \\
&= 
z_{r_{k-1}}^{-1}
\ldots
z_{r_1}^{-1}
z_{r_{|r|}}^{-1} x_{r_{|r|}}
\lambda_r. 
 \end{align*}
 \end{itemize}
Recall that, due to the way in which we chose the original group presentation $\Gpres{A}{R}$ for the finite group $G$, none of the generators $a \in A$ is equal to the identity of $G$.         

Moreover, all the defining relators are positive words over $A$. Finally, notice that no proper subword of a defining relator is equal to $1$ in $G$; indeed if it were then (since the relators are all of length $2$ or $3$) this would imply that a
single letter also equal to $1$ in $G$, but we chose our generating
set to exclude the identity element.

To prove that the monoid is {$\GreenR_1$-injective} we need in particular that all the elements in the set 
\[
\{
z_{r_{k-1}}^{-1}
\ldots
z_{r_1}^{-1}
z_{r_{|r|}}^{-1} x_{r_{|r|}}\lambda_r \mid
r \in R, 1 \leq k \leq |r|-1
\}
\]
are distinct in the group $K$. Because $K$ admits a decomposition as $FG(\Lambda \cup X) \ast G$ two such elements can clearly only be equal in $K$ if they
correspond to the same relator $r$, in other words, if they are
$z_{r_{i-1}}^{-1}
\ldots
z_{r_1}^{-1}
z_{r_{|r|}}^{-1} x_{r_{|r|}} \lambda_r$ and 
$z_{r_{j-1}}^{-1}
\ldots
z_{r_1}^{-1}
z_{r_{|r|}}^{-1} x_{r_{|r|}} \lambda_r$ 
for the same $r$ and different $i$ and $j$.
Assuming without loss of generality that $i > j$ and cancelling, we obtain
$z_{r_{i-1}}^{-1} \dots z_{r_{j}}^{-1} =1$.
Now again using the free product decomposition of $K$, it follows that
$r_j \dots r_{i-1} = 1$ in $G$. But this contradicts the fact established above that no proper subword of a defining relator in $G$ is equal to $1$ in $G$.

Next we claim that the submonoid of 
\[
K = 
FG(\Lambda \cup X) \ast 
\Gpres{Z}{\tilde{r}=1 \ (r \in R)} 
\]
generated by the set  
\begin{align*}
Q = \{
& x_{r_k}, \;
z_{r_k}^{-1} x_{r_k}, \; 
x_{r_{|r|}}\lambda_r \mid
r \in R, 1 \leq k \leq |r|
\} \\ 
& \cup 
\{
z_{r_{k-1}}^{-1}
\ldots
z_{r_1}^{-1}
z_{r_{|r|}}^{-1} x_{r_{|r|}}\lambda_r \mid
r \in R, 1 \leq k \leq |r|-1
\}
\end{align*}
of all the images of all the prefixes of relators in the presentation of $M$  
is a free monoid freely generated by these generators.  
These generators are all distinct by the argument above.
Now any product of these generators is in normal form with respect to the free product decomposition  
\[
K = FG(\Lambda \cup X) \ast \Gpres{Z}{\tilde{r}=1 \ (r \in R)}
= FG(\Lambda \cup X) \ast G
\]
From this together with the fact that every generator comes from the set $GX\Lambda$ we can deduce that the submonoid of $K$ generated by $Q$ is a free monoid with free generating set $Q$.
Now, by Proposition~\ref{prop_r1gen} the submonoid $\GreenR_1$ is generated as a monoid by the proper prefixes of the defining relators $s_{r,k}$.
Thus given two distinct elements $m$ and $n$ of $\GreenR_1$ they can each be written as 
a product of the elements
\[
x_{r_k}, \quad 
y_{r_k}^{-1}, \quad  
x_{r_k}\delta_{r,k}
\]  
where $r \in R$, $1 \leq k \leq |r|$.  
Since $m \neq n$ in the monoid these two products of prefixes must be distinct as words over 
$\{ x_{r_k},  
y_{r_k}^{-1},   
x_{r_k}\delta_{r,k}
  \}$.
This means that these products map to distinct products over $Q$ 
in $K$ 
from which it follows that $m$ and $n$ map to distinct elements of $K$, as the submonoid of $K$ generated by $Q$ is a free monoid with free generating set $Q$.
 \end{proof}
Notice that the argument in the proof of  
Lemma~\ref{lemma_misroi} actually shows that the submonoid $\GreenR_1$ of right units of $M$ is a free monoid of finite rank.

We now have all the ingredients needed to prove the main theorem of this section.
\begin{proof}[Proof of Theorem~\ref{thm:trivGroupUnits}]
For each finite group $G$ we have constructed a special inverse monoid $M$ which is \roi\ 
 (by Lemma~\ref{lemma_misroi}), 
has trivial group of units (by Lemma~\ref{lemma_sgoprops}),
and has a Sch\"utzenberger graph isomorphic (by
 Lemma~\ref{lemma_omegaisomorphism}) to the graph $\Omega$, which has automorphism group isomorphic to $G$
 (by Lemma~\ref{lem:AutOmega}). The result now follows from the fact that maximal subgroups of $M$ are precisely the
 automorphism groups of the Sch\"utzenberger graphs \cite[Theorem 3.5]{Stephen:1990ss}.
 \end{proof}

\begin{remark} 
The inverse monoids constructed in the proof of Theorem~\ref{thm:trivGroupUnits} are not in general $E$-unitary. 
Indeed, let $M$ be the inverse monoid constructed in Example~\ref{Ex:R1ConstructionExample}. 
Two of the relators in the presentation of that inverse monoid are 
\[
x_{a} \delta_{r,1} (\delta_{r,3})^{-1} y_{a}=1 
\ \ \mbox{and} \ \
x_{a} \delta_{r,3} (\delta_{r,2})^{-1} y_{a}=1.
\]
It follows that in the maximal group image $K$ we have  
\[
x_{a} \delta_{r,1} (\delta_{r,3})^{-1} y_{a}= 
x_{a} \delta_{r,3} (\delta_{r,2})^{-1} y_{a}
\]
from which 
it follows that 
$\delta_{r,1}\delta_{r,3}^{-1}\delta_{r,2}\delta_{r,3}^{-1}=1$ in $K$. 
Lemma~\ref{lemma_sgoprops}(i) and points three and four in Proposition~\ref{prop_greenschutz} together show that every element of $\Delta$ is neither a left unit not a right unit.   
Using this, it can then be shown by constructing the Sch\"utzenberger graph $S\Gamma(\delta_{r,1}\delta_{r,3}^{-1}\delta_{r,2}\delta_{r,3}^{-1})$ (or alternatively by applying
 Theorem~\ref{Thm:NewFree} below, which tells us that $\Delta$ generates a free inverse submonoid of $M$) that the word $\delta_{r,1}\delta_{r,3}^{-1}\delta_{r,2}\delta_{r,3}^{-1}$ does not represent an idempotent in $M$, so $M$ is not $E$-unitary.

Variations of this argument show that the inverse monoids constructed in the proof of Theorem~\ref{thm:trivGroupUnits} are never
$E$-unitary. Indeed, if the finite group $G$ has a pair of non-identity elements $g,h \in G$ such that $gh \neq 1$ then a set $R$ of defining relators
for $G$ contains distinct relators $gha$ and $agh$ for some $a$, from which we obtain two distinct defining relators for $M$ which begin with
$x_h$ and ending with $y_g$, and apply a similar argument to that above. If $G$ does not have such a pair of elements then it must be isomorphic to
$\mathbb{Z}_2$ or $\mathbb{Z}_3$. The case of $\mathbb{Z}_3$ is exactly the above example, while in the case of $\mathbb{Z}_2$ the set $R$ contains
a relator for $G$ of the form $aa$, which yields two relators for $M$ beginning $x_a$ and ending $y_a$, to which a similar argument can again be applied.
\end{remark}

We do not know how to construct an $E$-unitary special inverse monoid with trivial group of units and non-trivial maximal subgroups, or even just with
a maximal subgroup which does not embed into the group of units. We ask if this is possible.

\begin{question}
Do the maximal subgroups of an $E$-unitary special inverse monoid necessarily embed into the group of units?
\end{question}

\section{Generators which are not right or left units}\label{sec_nonroi}

In this final section we note the following rather surprising theorem, which allows us in particular to easily construct examples of finitely presented special inverse monoids with many different non-isomorphic maximal subgroups. Indeed it suggests that this kind of behaviour is in some sense ``the norm'' for special inverse monoids! This contrasts sharply with the case of special (non-inverse) monoids, where by a result of Malheiro \cite{Malheiro2005} all maximal subgroups lie in the $\GreenD$-class of $1$ and therefore are necessarily isomorphic. 

\begin{theorem}\label{thm_rubbishdichotomy} Let $M$ be a special inverse monoid defined by a presentation with a generator which represents neither a left nor a right unit. Then $M$ contains every finite subgroup of the group of units as a maximal subgroup.
\end{theorem}
\begin{proof}Choose a generator $v$ which is neither a left nor a right unit. Let $Q$ be a finite subgroup of the group of units, say $|Q| = k$, and let $u_1, \dots, u_k$ be words representing the elements of $Q$. Consider the element
$$w \ = \ \prod_{i = 1}^k u_i v v' u_i',$$
of $M$, noting that the order of the product is unimportant because the factors are idempotent and therefore commute. We shall describe the Sch\"utzenberger graph $\sgw$.

By Stephen's procedure (see Remark~\ref{rem:TreeOfBalls}), it is easy to see that $\sgw$ can be obtained by starting with $\sgo$, we shall call this the \textit{central subgraph} and we shall denote it by $C$, and for each vertex corresponding to an element of $Q$, gluing on an edge leaving it labelled $v$ and a new copy of $\sgo$ rooted at the far end of the edge, and bi-determinising. We claim that in fact this graph is already bi-deterministic, so that no bi-determinising is needed. Indeed, clearly no folding can take place within the copies of $\sgo$ since $\sgo$ is already bi-deterministic. The new $v$-edges cannot fold with each other because they do not share any endpoints. It remains only to show that the new $v$-edges cannot fold into an edge in one of the copies of $\sgo$. Notice that the new $v$-edges all connect at both ends into vertices of $\sgo$ corresponding to elements of the group of units. If there was an existing $v$-edge in $\sgo$ at one of these vertices for one of the new $v$-edges to fold into, it would therefore follow that $v$ is either a right unit or a left unit (depending on the orientation of the edge), giving a contradiction.

Recall that each preblock $P$ is a subgraph of $S\Gamma(w)$ that is isomorphic to $S\Gamma(1)$. Next we claim that for every preblock $P$ of $S\Gamma(w)$ either $P$ is contained in $C$ or else $P$ and $C$ have disjoint vertex sets. Indeed, let $P$ be a preblock. Choose some root $p$ of $P$. If $p$ is contained in $C$ then since $C$ is a copy of $S\Gamma(1)$ in $S\Gamma(w)$ it follows that $P$ is entirely contained in $C$. Now suppose $p$ lies outside $C$. If $P$ did intersect $C$ then we could choose a path of minimal possible length from a vertex in the intersection of $C$ and $P$ to $p$. By construction of the graph, this path must be labelled by a word of the form $vh$ where the $v$ traverses a $v$-edge from the Munn tree. By the construction of the graph, $h$ labels a path starting at the root inside a glued-on copy of $S\Gamma(1)$, and hence is right invertible. On the other hand, $vh$ labels a path inside the preblock $P$ ending at the root $p$ of $P$, and hence both $vh$ and $h$ are left invertible. But now conjugating $vh$ by the unit $h$ we again see that $hv$ is left invertible and thus $v$ is left invertible, which is a contradiction. This completes the proof of the claim.

We say that a preblock $P$ has property (*) if $P$ is \emph{not} contained in a preblock $P'$ such that $P'$ has a $v$-edge coming into its root. From the previous paragraph, and the construction of $S\Gamma(w)$, it follows that every preblock $P$ with property (*) must be contained in the central subgraph $C$. We claim that the central subgraph $C$ itself has property (*). Indeed, otherwise $C$ would be contained in some preblock $P$ where $P$ has a $v$-edge coming in. The root of $P$ must be in $C$ by the claim in the previous paragraph. But this implies that $P$ is contained in $C$ and thus $P=C$. If $e$ is a $v$-edge coming into the root of $P$ then from the structure of $S\Gamma(w)$ given in the second paragraph of the proof above, the initial vertex of $e$ cannot be outside of $P=C$, thus the edge $e$ must be contained in $P=C$. But now the isomorphism from $P$ to the central subgraph $C$ mapping the root $p$ of $P$ to the root of $C$ (which exists because they are both isomorphic to $S\Gamma(1)$) gives an automorphism of the central subgraph $C$ taking $p$ to the root of $C$. But considering the image of the edge $e$ under this automorphism, this implies that the root of $C$ has a $v$-edge coming into it, where that edge is contained in $C$. But $C$ is a copy of $S\Gamma(1)$, so this means that in $S\Gamma(1)$ there is a $v$-edge coming into the root. This implies that $v$ is left invertible, which is a contradiction.               

It follows from the previous paragraph that if we consider the collection all preblocks $P$ with property (*) ordered by subset inclusion then the central subgraph $C$ is the unique maximal element of this collection of preblocks. That in turn implies that automorphisms of $\sgw$ must restrict to automorphisms of the central subgraph $C$. Now it is easy to see that the automorphisms of $\sgw$ are exactly those automorphisms of the central subgraph which fix the set of vertices corresponding to $Q$, in other words, the automorphisms corresponding to elements of $Q$. Thus, the maximal subgroups in the $\GreenD$-class of $w$ are isomorphic to $Q$ as required.
\end{proof}

\begin{remark}
The proof of Theorem~\ref{thm_rubbishdichotomy} is clearly reminiscent of our reasoning with blocks in the $\GreenR_1$-injective case above; indeed the
similarity is our reason for including the result in this article, the main focus of which is otherwise on the \roi\ case. Since we
are not here assuming $\GreenR_1$-injectivity we do not have access to the machinery above to guarantee a block decomposition for every Sch\"utzenberger
graph, but it just so happens that the particular Sch\"utzenberger graph $\sgw$ constructed in the proof does have a block decomposition.
\end{remark}

As a consequence of Theorem~\ref{thm_rubbishdichotomy} we obtain the following corollary, which is a very slight strengthening (because the free inverse monoid has
rank $1$, rather than $2$) of a result which was established by other means in our recent work \cite[Corollary 5.11]{GrayKambitesHClasses}.

\begin{corollary}
There exists an $E$-unitary finitely presented special inverse monoid, which is the free product of a group and a free inverse monoid of rank $1$ and which has every finite group as a maximal subgroup.
\end{corollary}
\begin{proof}
Take a finitely presented group $G$ having every finite group as a subgroup (for example, Higman's universal group, which contains every
finitely 
\textit{presented} group \cite[Theorem 7.3]{Lyndon:2001lh}), and consider the inverse monoid free product of $G$ with a free inverse monoid of rank one. That the monoid is $E$-unitary follows easily from the fact that groups and free inverse monoids are both $E$-unitary. The result now follows from
Theorem~\ref{thm_rubbishdichotomy}.
\end{proof}

Our final result gives sufficient conditions for a subset of generators to generate a free inverse monoid. 

\begin{theorem}\label{Thm:NewFree}
Let $M = \pres{X}{R}$ be a special inverse monoid and let $A$ be a subset of $X$ containing neither left nor right units.
Then the inverse submonoid of $M$ generated by $A$ is isomorphic to the free inverse monoid on $A$.    
  \end{theorem}
\begin{proof}
Consider two words $u$ and $w$ over $A^{\pm 1}$ that are not equal in the free inverse monoid on $A$. Let $\Gamma_u$ and $\Gamma_w$ be their Munn
trees, that is, the graphs obtained by bi-determinising the lines labelled by $u$ and $v$ respectively.
The Sch\"utzenberger graph $S\Gamma(u)$ [respectively, $S\Gamma(w)$] is obtained by attaching copies of $S\Gamma(1)$ to every vertex of $\Gamma_u$ [respectively, $\Gamma_w$] and bi-determinising.
By Proposition~\ref{prop_greenschutz}, the fact that no element of $A$ is a left or a right unit means that the root vertex of $S\Gamma(1)$ is not
incident with any edges labelled by letter in $A$. It follows that the graph given by attaching a copy of $S\Gamma(1)$ to each vertex of $\Gamma_u$ [respectively, $\Gamma_w$] is already bi-determinised and hence is already equal to $S\Gamma(u)$ [respectively, $S\Gamma(w)$].     

Now if the Munn trees $\Gamma_u$ and $\Gamma_w$ are not isomorphic as rooted graphs then (by swapping $u$ and $w$ if necessary) we may assume without loss of generality that $\Gamma_u$ does not embed into $\Gamma_w$ as a rooted subgraph. It follows that $u$ cannot be read from the root in $\Gamma_w$. Since by the previous paragraph $S\Gamma(w)$ is obtained from $\Gamma_w$ by attaching graphs which have no edges from $A$ incident with the root, it follows that $u$ also cannot be read from the root of $S\Gamma(w)$ and thus $u$. Hence, $u$ and $w$ are not $\gr$-related in $M$ and thus not equal in $M$.       

The remaining case is that $\Gamma_u$ and $\Gamma_w$ are isomorphic as rooted graphs. In this case, since $u \neq w$ in the free inverse monoid it follows that reading $w$ from the root in $\Gamma_u$ leads to a different terminal vertex than reading $u$ in $\Gamma_u$. Since by the previous paragraph $\Gamma_u$ embeds in
$S\Gamma(u)$, it follows that $uu^{-1}u \neq uu^{-1}w$ in $M$ and thus $u \neq w$ in $M$.

It follows that the natural homomorphism from the free inverse monoid on $A$ to $M$ is injective, as required.    
  \end{proof}

\section*{Acknowledgements}
The authors thank the anonymous referees for their many suggestions which helped to improve the exposition.

\end{document}